\documentclass[pdflatex,sn-basic,Numbered]{sn-jnl}

\usepackage{amsmath,amssymb,mathtools}
\usepackage{algorithm}
\usepackage{algpseudocode}
\usepackage{array}
\usepackage{booktabs}
\usepackage{enumitem}
\usepackage{graphicx}
\usepackage{microtype}
\usepackage{xcolor}

\let\table\tableorg
\let\endtable\endtableorg
\hypersetup{hypertexnames=false}

\theoremstyle{thmstyleone}
\newtheorem{theorem}{Theorem}[section]
\newtheorem{lemma}[theorem]{Lemma}
\newtheorem{proposition}[theorem]{Proposition}
\newtheorem{corollary}[theorem]{Corollary}
\theoremstyle{thmstylethree}
\newtheorem{definition}[theorem]{Definition}
\newtheorem{assumption}{Assumption}
\theoremstyle{thmstyletwo}
\newtheorem{remark}[theorem]{Remark}

\newcommand{\M}{\mathcal M}
\newcommand{\X}{\mathcal X}
\newcommand{\Xstar}{\mathcal X^\star}
\newcommand{\R}{\mathbb R}
\newcommand{\Exp}{\operatorname{Exp}}
\newcommand{\Log}{\operatorname{Log}}
\newcommand{\grad}{\operatorname{grad}}

\newcommand{\diam}{\operatorname{diam}}
\newcommand{\inj}{\operatorname{inj}}
\newcommand{\Dset}{\mathcal V}
\newcommand{\Xlo}{\underline X}
\newcommand{\Xup}{\overline X}
\newcommand{\inner}[2]{\left\langle #1,#2\right\rangle}
\newcommand{\norm}[1]{\left\lVert #1\right\rVert}

\begin{document}

\title[Fast Rates for Open-Loop Riemannian Frank--Wolfe]{Open-Loop Riemannian Frank--Wolfe: Fast Rates under Error Bounds and Scaling Inequalities}

\author*[1]{\fnm{Kangming} \sur{Chen}}\email{kangming@tmu.ac.jp}

\affil*[1]{\orgdiv{Faculty of Economics and Business Administration},
\orgname{Tokyo Metropolitan University}, \orgaddress{\street{1-1 Minami-Osawa}, \city{Hachioji}, \postcode{192--0397}, \state{Tokyo}, \country{Japan}}}
\abstract{%

We explore fast convergence of the Riemannian Frank--Wolfe method for smooth geodesically convex optimization over compact feasible sets.  Hadamard manifolds are the main setting. On general complete manifolds, the analysis accounts for all feasible minimizing geodesics. We consider the open loop step-size $\eta_k=a/(k+a)$, which only uses the iteration index.  Under a local H\"olderian error bound and local length-normalized directional scaling, every $a>2$ gives the eventual rate $O(k^{-1/(1-\theta)})$ for $\theta\in(0,1/2]$. An interior-ball condition yields $O(k^{-2})$ for strongly geodesically convex objectives. Under an exact Riemannian scaling inequality and a uniform positive lower bound on the gradient norm, every $a\geq2$ gives $O(k^{-a})$ after an explicit threshold index. The same rate holds for the smallest Frank--Wolfe gap over the most recent half of the iterates.
For geodesic balls of radius \(R<\pi/2\) in the unit sphere, we establish the scaling inequality with $\alpha_R=\tfrac12\cot R$, yielding $O(k^{-a})$ primal error and recent-window gap rates. We also analyze the standard gap-feedback short step under the local error-bound conditions, obtaining $O(k^{-1/(1-2\theta)})$ for $\theta<1/2$ and a linear rate for $\theta=1/2$.  Numerical experiments illustrate the predicted rates and compare iteration-only and feedback-based step selection.
}

\keywords{Riemannian optimization, Frank--Wolfe method, open-loop step sizes, geodesic convexity, H\"olderian error bounds }
\pacs[MSC Classification]{ 90C25, 90C30, 53C22}

\maketitle

\section{Introduction}
\label{sec:introduction}

We consider smooth geodesically convex optimization over a compact feasible set on a Riemannian manifold.
Frank--Wolfe methods avoid projections by solving a linear minimization subproblem \citep{frank1956algorithm,jaggi2013revisiting}.
Their Riemannian extensions replace Euclidean search directions with feasible minimizing geodesics and have been studied in~\cite{weber2023riemannian} and subsequent works~\citep{chen2026riemannian,scieur2026strongly}.

Open-loop step-size rules for conditional-gradient methods date back to \citet{dunn1978conditional}.  More recently,~\citet{wirth2023acceleration} showed that the Frank--Wolfe method with predetermined step sizes can achieve faster rates under suitable local conditions.  ~\citet{wirth2025affine} later derived affine-invariant bounds for fixed-parameter schedules, including an eventual $O(k^{-\ell})$ rate for $\eta_k=\ell/(k+\ell)$ under the strong $(M,1)$-growth condition.
A log-adaptive open-loop rule that removes the need to choose $\ell$ in advance  subsequently proposed in~\cite{wirth2025adaptive}.
By contrast, feedback strategies use current oracle or problem information.
Under strict relative-interior conditions, classical Frank--Wolfe can attain a linear rate~\citep{guelat1986comments}. Another standard choice is the short-step rule, which selects the step size from a quadratic upper model along the current Frank--Wolfe direction~\citep{pokutta2024short}. Faster Euclidean Frank--Wolfe rates have also been linked to feasible-set geometry.  Garber and Hazan proved an \(O(1/k^2)\) rate for smooth strongly convex
objectives over strongly convex feasible sets, while Kerdreux et al. extended this perspective to uniformly convex sets and affine-invariant scaling conditions~\citep{garber2015faster,kerdreux2021affine,kerdreux2021uniform}.

For Riemannian Frank--Wolfe, \citet{weber2023riemannian} established the standard $O(1/k)$ rate and a feedback-based linear rate under strong geodesic convexity and strict interiority, while~\citet{scieur2026strongly} proved a global linear rate for a short step under a scaling inequality and a gradient lower bound.  Both fast-rate results use current-iterate information.
Retraction-based generalized conditional-gradient methods instead define the oracle and update through a retraction and its local inverse \citep{chen2026riemannian}.
It remains open whether the standard RFW update can achieve a rate faster than $O(1/k)$ with a predetermined step size.

Our first result is an intrinsic counterpart of the Euclidean open-loop analysis of \citet{wirth2023acceleration,wirth2025affine}.  Exact oracle solutions and minimizing geodesics need not be unique, so we formulate length-normalized directional scaling using the largest length among all exact oracle solutions.  The resulting estimate therefore holds for every admissible oracle selection.  Combined with a local H\"olderian error bound (HEB), it gives the eventual rate $O(k^{-1/(1-\theta)})$ for every $a>2$ and $\theta\in(0,1/2]$.  Compactness and the global $O(1/k)$ bound ensure finite entry into the local region.  An interior ball verifies directional scaling, and strong geodesic convexity then gives $O(k^{-2})$.

A second mechanism uses the exact Riemannian scaling inequality of \citet{scieur2026strongly}.
Together with a uniform positive lower bound on the gradient norm, it yields the strong gap-growth estimate of \citet{wirth2025affine}, uniformly over exact oracle solutions.
We obtain $O(k^{-a})$ after an explicit threshold index, both for the primal error and for the minimum Frank--Wolfe gap over the most recent half of the iterates.
A direct geometric argument shows that every spherical ball of radius $R<\pi/2$ satisfies the scaling inequality with $\alpha_R=\tfrac12\cot R$ in the unit sphere.
For comparison, we also analyze the standard gap-feedback short step under the local error-bound conditions.
Table~\ref{tab:related-work} summarizes the parameter ranges, assumptions, and rates.

\begin{table}[htbp]
\centering
\caption{Representative primal convergence rates for  Frank--Wolfe
methods.  For comparison, all rows are stated for smooth convex objectives over compact feasible sets; the table lists only the additional conditions that distinguish the rates.  An ``eventual'' rate holds after entry into the
corresponding local regime.}
\label{tab:related-work}
\footnotesize
\setlength{\tabcolsep}{3pt}
\begin{tabular}{@{}
  >{\raggedright\arraybackslash}p{0.22\textwidth}
  >{\raggedright\arraybackslash}p{0.20\textwidth}
  >{\raggedright\arraybackslash}p{0.38\textwidth}
  >{\raggedright\arraybackslash}p{0.14\textwidth}@{}}
\toprule
Method and geometry & Step rule & Additional conditions & Primal rate \\
\midrule
Euclidean FW \citep{wirth2023acceleration}
 & open loop, $4/(k+4)$
 & global HEB; unique relative-interior minimizer
 & eventual $O(k^{-1/(1-\theta)})$ \\

Euclidean FW \citep{wirth2025affine}
 & open loop, $\ell/(k+\ell)$, $\ell\in\mathbb N$
 & strong $(M,1)$-growth
 & eventual $O(k^{-\ell})$ \\

RFW \citep{weber2023riemannian}
 & open loop, $2/(k+2)$
 & none beyond the standing assumptions
 & global $O(1/k)$ \\

RFW \citep{weber2023riemannian}
 & primal-gap feedback using problem constants
 & strong geodesic convexity; optimizer in an interior ball
 & global linear \\

RFW \citep{scieur2026strongly}
 & gap-based short step
 & global Riemannian scaling inequality;
   $\inf_{x\in\mathcal X}\|\grad f(x)\|_x>0$
 & global linear \\
\midrule
RFW-OL($a$) (this paper)
 & open loop, $a/(k+a)$, $a>2$
 & local HEB; local length-normalized directional scaling
 & eventual $O(k^{-1/(1-\theta)})$ \\

RFW-OL($a$) (this paper)
 & open loop, $a/(k+a)$, $a\geq2$
 & global Riemannian set scaling ;
   $\inf_{x\in\mathcal X}\|\grad f(x)\|_x>0$
 & eventual $O(k^{-a})$ \\

RFW short step (this paper)
 & feedback using $g_k$, $\ell_k$, and $\widehat L$
 & local HEB; local length-normalized directional scaling
 & eventual $O(k^{-1/(1-2\theta)})$ if $\theta<1/2$;
   eventual linear if $\theta=1/2$ \\
\bottomrule
\end{tabular}
\end{table}

The paper is organized as follows.
Section~\ref{sec:setting} introduces the geometric setting and RFW oracle.
Section~\ref{sec:algorithm} describes the open-loop method and proves its baseline rate, after which Sections~\ref{sec:acceleration} and~\ref{sec:geometry} develop the two fast-rate analyses and verify their geometric conditions.  Numerical results are presented in Section~\ref{sec:numerics}.
Section~\ref{sec:conclusion} concludes the paper.

\section{Problem setting and the Riemannian Frank--Wolfe oracle}
\label{sec:setting}

Let $(\M,  g)$ be a connected, complete, finite-dimensional Riemannian manifold with Riemannian distance $d$.  For $x\in\M$, let $T_x\M$ denote the tangent space, with inner product $\inner{\cdot}{\cdot}_x$ and associated norm $\norm{\cdot}_x$.  The Riemannian gradient is defined by $\mathrm{d}f(x)[\xi]=\inner{\grad f(x)}{\xi}_x$ for all $\xi\in T_x\M$. For $v\in T_x\M$, $t\mapsto\Exp_x(tv)$ is the geodesic starting from $x$ with initial velocity $v$.
A Hadamard manifold is complete, simply connected, and has nonpositive sectional curvature.  By the Cartan--Hadamard theorem, $\Exp_x$ is a global diffeomorphism for every $x\in\M$.  Hence $\Log_x:=\Exp_x^{-1}$ is globally defined, and any two points are joined by a unique minimizing geodesic~\citep{sak1996riemannian}.

The analysis is formulated on a general complete manifold by allowing all feasible minimizing geodesics, while the Hadamard case remains the main algorithmic and numerical setting.
Let $\X\subset\M$ be nonempty and compact, and set
\[
  D:=\diam(\X)=\sup_{x,y\in\X}d(x,y)<\infty.
\]
The case $D=0$ is trivial, so we  assume $D>0$ in this paper.
By the Hopf--Rinow theorem, completeness guarantees a minimizing geodesic between any two points of $\M$~\citep{sak1996riemannian}, but whether such a geodesic remains in $\X$ is not guaranteed.  We therefore impose the following feasibility condition.

\begin{definition}[Weak minimizing-geodesic convexity]
The set $\X$ is \emph{weakly minimizing-geodesically convex} if, for every $x,y\in\X$, there exists at least one distance-minimizing geodesic from $x$ to $y$ whose image is contained in $\X$.
\end{definition}

Throughout this paper, we assume that $\X$ is weakly minimizing-geodesically convex.
To avoid fixing a global choice of logarithmic map, for each $x\in\X$ we consider the set of feasible minimizing initial velocities
\begin{equation}
\Dset(x):=
\left\{
v\in T_x\M:
\begin{array}{l}
\Exp_x(v)\in\X,\quad \norm{v}_x=d(x,\Exp_x(v)),\\
\Exp_x(tv)\in\X\quad\text{for all }t\in[0,1]
\end{array}
\right\}.
\label{eq:feasible-directions}
\end{equation}
If $\M$ is Hadamard, then $\Dset(x)=\{\Log_x(y):y\in\X\}$.
On a general complete manifold, $\Dset(x)$ instead retains the initial velocities of all feasible minimizing geodesics, without requiring a single-valued global logarithmic map.
\begin{lemma}
\label{lem:direction-compact}
For every $x\in\X$, the set $\Dset(x)$ is nonempty and compact in $T_x\M$.  Moreover,
$
\norm{v}_x\leq D
$ for every $v\in\Dset(x).$
\end{lemma}

\begin{proof}
Since $\Exp_x(0)=x$, we have $0\in\Dset(x)$.  For any
$v\in\Dset(x)$,
\[
 \norm{v}_x=d(x,\Exp_x(v))\leq D,
\]
so $\Dset(x)$ is bounded.
Let $(v_j)_{j\geq1}\subseteq\Dset(x)$ satisfy $\norm{v_j-v}_x\to0$.  For any fixed $t\in[0,1]$, continuity of $\Exp_x$ and closedness of $\X$ give
\[
 \Exp_x(tv)=\lim_{j\to\infty}\Exp_x(tv_j)\in\X.
\]
Moreover,
\[
 \norm{v}_x
 =\lim_{j\to\infty}\norm{v_j}_x
 =\lim_{j\to\infty}d(x,\Exp_x(v_j))
 =d(x,\Exp_x(v)).
\]
Hence $v\in\Dset(x)$, and $\Dset(x)$ is closed.  Compactness follows because $T_x\M$ is finite dimensional.
\end{proof}

We consider
\begin{equation}
  \min_{x\in\X} f(x),
  \label{eq:problem}
\end{equation}
where $f$ is differentiable on an open neighborhood of $\X$.  Define
$$
 f^\star:=\min_{x\in\X}f(x),\qquad
 \Xstar:=\arg\min_{x\in\X}f(x),\qquad
 h(x):=f(x)-f^\star.
$$
Since $f$ is continuous and $\X$ is compact, the minimum is attained and
$\Xstar\neq\varnothing$.

\begin{assumption}[Convexity and a quadratic upper model on admissible geodesics]
\label{ass:regularity}
For every $x\in\X$, $v\in\Dset(x)$, and $t\in[0,1]$,
\begin{align}
 f(\Exp_x(tv))
 &\leq (1-t)f(x)+t f(\Exp_x(v)),
 \label{eq:gconvex}\\
 f(\Exp_x(tv))
 &\leq f(x)+t\inner{\grad f(x)}{v}_x
       +\frac{L}{2}t^2\norm{v}_x^2
 \label{eq:gsmooth}
\end{align}
for a constant $L>0$ independent of $x$, $v$, and $t$.
\end{assumption}
Both inequalities are standard in Riemannian optimization.  The first is the geodesic-convexity inequality, while the second is the quadratic upper model
associated with geodesic smoothness; see
\citet[Definitions~1 and~4]{zhang2016first} and
\citet[Section~2.1]{weber2023riemannian}.
The conditions are imposed along every admissible minimizing geodesic, so the
analysis does not depend on a particular geodesic selection.
The geodesic-convexity inequality implies
\begin{equation}
 f(\Exp_x(v))
 \geq f(x)+\inner{\grad f(x)}{v}_x,
 \qquad v\in\Dset(x).
 \label{eq:first-order-convexity}
\end{equation}

\begin{remark}[Relation to standard smoothness]
The quadratic upper model follows from an $L$-Lipschitz Riemannian gradient
along admissible minimizing geodesics, where gradients are compared by
parallel transport.  For $f\in C^2$, a sufficient condition is
$\|\operatorname{Hess}f\|_{\rm op}\leq L$ along these geodesics
\citep[Section~10.4]{boumal2023introduction}.
\end{remark}

\begin{definition}[Strong convexity along admissible minimizing geodesics]
\label{def:strong-gconvexity}
For $\mu>0$, the objective $f$ is \emph{$\mu$-strongly geodesically convex
along the admissible minimizing geodesics in $\X$} if, for every $x\in\X$,
$v\in\Dset(x)$, and $t\in[0,1]$,
\begin{equation}
 f(\Exp_x(tv))
 \leq (1-t)f(x)+t f(\Exp_x(v))
       -\frac{\mu}{2}t(1-t)\norm{v}_x^2.
 \label{eq:strong-gconvexity}
\end{equation}
In particular, differentiation at $t=0$ gives
\begin{equation}
 f(\Exp_x(v))
 \geq f(x)+\inner{\grad f(x)}{v}_x
       +\frac{\mu}{2}\norm{v}_x^2.
 \label{eq:first-order-strong-convexity}
\end{equation}
The quantifier over $v\in\Dset(x)$ requires the inequality along every
admissible feasible minimizing geodesic in \eqref{eq:feasible-directions}.
\end{definition}

The exact RFW oracle is the intrinsic counterpart of the Euclidean linear minimization oracle (LMO)~\cite{weber2023riemannian}.
At $x\in\X$, its set-valued form is
\begin{equation}
  v(x)\in\arg\min_{v\in\Dset(x)}
  \inner{\grad f(x)}{v}_x.
  \label{eq:rfw-oracle}
\end{equation}
Define its compact solution set and maximal solution length by
\begin{equation}
 \mathcal S(x):=\arg\min_{v\in\Dset(x)}
 \inner{\grad f(x)}{v}_x,
 \qquad
 \bar\ell(x):=\max_{v\in\mathcal S(x)}\norm{v}_x.
 \label{eq:oracle-solution-set}
\end{equation}
By Lemma~\ref{lem:direction-compact}, $\Dset(x)$ is compact, and hence the
continuity of the linear objective implies that $\mathcal S(x)$ is nonempty
and compact.  Distinct elements of $\mathcal S(x)$ may correspond to different
endpoints or, outside the Hadamard setting, to different minimizing geodesics
with the same endpoint.
Define
$$
  g(x):=
  -\min_{v\in\Dset(x)}
  \inner{\grad f(x)}{v}_x .
$$
For any selection $v(x)\in\mathcal S(x)$,
\[
  g(x)=-\inner{\grad f(x)}{v(x)}_x,
  \qquad
  \ell(x):=\norm{v(x)}_x.
\]
Thus the gap $g(x)$ is independent of the oracle selection, whereas $\ell(x)$ may depend on it.  In particular,
\[
  \ell(x)\leq\bar\ell(x)\leq D.
\]

\section{RFW with open-loop step sizes}
\label{sec:algorithm}

We consider the Riemannian Frank--Wolfe method with the open-loop schedule $\eta_k=a/(k+a)$, $a\ge2$.  On a Hadamard manifold, our oracle formulation coincides with that of \citet{weber2023riemannian}, and $a=2$ gives the standard schedule.

\begin{algorithm}[H]
\caption{RFW with the $a/(k+a)$ open-loop schedule (RFW-OL($a$))}
\label{alg:ol-rfw}
\begin{algorithmic}[1]
\Require $x_0\in\X$ and $a\geq2$
\For{$k=0,1,2,\ldots$}
  \State $v_k\in\arg\min_{v\in\Dset(x_k)}
         \inner{\grad f(x_k)}{v}_{x_k}$
  \State $g_k\gets-\inner{\grad f(x_k)}{v_k}_{x_k}$
  \If{$g_k=0$}
    \State \Return $x_k$
  \EndIf
  \State $\eta_k\gets a/(k+a)$
  \State $x_{k+1}\gets\Exp_{x_k}(\eta_k v_k)$
\EndFor
\end{algorithmic}
\end{algorithm}
We first give feasibility and the associated gap result.
\begin{lemma}
\label{lem:gap-domination}
Suppose Assumption~\ref{ass:regularity} holds.  Then
\begin{equation}
  g(x)\geq h(x),\qquad x\in\X.
  \label{eq:gap-dominates}
\end{equation}
Moreover, $g(x)=0$ if and only if $x\in\Xstar$.  Every iterate generated by Algorithm~\ref{alg:ol-rfw} belongs to $\X$, and hence
\[
  g_k=g(x_k)\geq h_k:=h(x_k).
\]
\end{lemma}

\begin{proof}
Fix $x\in\X$ and $x^\star\in\Xstar$.  Weak minimizing-geodesic convexity gives a direction $v^\star\in\Dset(x)$ such that $\Exp_x(v^\star)=x^\star$.  For any $v\in\mathcal S(x)$, oracle optimality and \eqref{eq:first-order-convexity} yield
\[
 \inner{\grad f(x)}{v}_x
 \leq \inner{\grad f(x)}{v^\star}_x
 \leq f^\star-f(x)=-h(x).
\]
Since $g(x)=-\inner{\grad f(x)}{v}_x$, this proves \eqref{eq:gap-dominates}.

Now suppose that $x\in\Xstar$.  For any $v\in\Dset(x)$, the function $\phi(t):=f(\Exp_x(tv))$ is convex on $[0,1]$ and attains its minimum at $t=0$.  Therefore
\[
 \inner{\grad f(x)}{v}_x=\phi'(0)\geq0.
\]
Because $0\in\Dset(x)$, the minimum oracle value is zero, and hence $g(x)=0$.  Conversely, if $g(x)=0$, then \eqref{eq:gap-dominates} gives $0\leq h(x)\leq g(x)=0$, so $x\in\Xstar$.

\end{proof}

The quadratic upper model yields the following one-step estimate.
\begin{lemma}
\label{lem:progress}
Suppose Assumption~\ref{ass:regularity} holds.  Fix $x\in\X$ and choose $v\in\mathcal S(x)$.  Set
\[
g:=-\inner{\grad f(x)}{v}_x,
\qquad
\ell:=\norm{v}_x.
\]
Then, for every $\eta\in[0,1]$, the point $x^+:=\Exp_x(\eta v)$ belongs to $\X$ and satisfies
\begin{equation}
 h(x^+)\leq h(x)-\eta g+\frac{L}{2}\eta^2\ell^2.
 \label{eq:generic-progress}
\end{equation}
Specifically, let
$
\ell_k:=\norm{v_k}_{x_k}.
$ Consequently, at every nonterminal iteration of
Algorithm~\ref{alg:ol-rfw},

\begin{align}
 h_{k+1}
 &\leq h_k-\eta_k g_k+\frac{L}{2}\eta_k^2\ell_k^2
 \label{eq:progress}\\
 &\leq (1-\eta_k)h_k+\frac{L}{2}\eta_k^2D^2.
 \label{eq:baseline-recurrence}
\end{align}
\end{lemma}

\begin{proof}
Since $v\in\Dset(x)$ and $\eta\in[0,1]$, we have
$x^+=\Exp_x(\eta v)\in\X$.  Applying \eqref{eq:gsmooth} with $t=\eta$ and substituting the definitions of
$g$ and $\ell$ yields \eqref{eq:generic-progress}.  The first recurrence follows by taking
$(x,v,\eta)=(x_k,v_k,\eta_k)$, and the second follows from
$g_k\geq h_k$ and $\ell_k\leq D$.
\end{proof}

Now we can derive the standard global rate.
\begin{theorem}[Global $O(1/k)$ rate]
\label{thm:baseline}
Let $a\geq2$ and suppose Assumption~\ref{ass:regularity} holds.  Every generated iterate $x_k$ with $k\geq1$ satisfies
\begin{equation}
  h_k\leq\frac{a^2LD^2}{2(k+a-1)}.
  \label{eq:baseline-rate}
\end{equation}
If Algorithm~\ref{alg:ol-rfw} terminates, its terminal point is an optimizer.
\end{theorem}

\begin{proof}
Set $Q:=a^2LD^2/2$.  If $x_1$ is generated, then $\eta_0=1$ and \eqref{eq:baseline-recurrence} gives
$
h_1\leq\frac{LD^2}{2}\leq\frac{Q}{a}.
$

Suppose that $x_{k+1}$ is generated and that
$h_k\leq Q/(k+a-1)$ for some $k\geq1$.  Since
$\eta_k=a/(k+a)$, \eqref{eq:baseline-recurrence} yields
\[
 \frac{h_{k+1}}{Q}
 \leq \frac{k}{(k+a)(k+a-1)}
      +\frac{1}{(k+a)^2}
 \leq \frac{1}{k+a},
\]
where the last inequality follows from
$
(a-2)(k+a)+1\geq0.
$
Thus \eqref{eq:baseline-rate} follows by induction.
If the algorithm terminates, its terminal point is optimal by Lemma~\ref{lem:gap-domination}.
\end{proof}

\section{Fast rates under error bounds and length-normalized directional scaling}
\label{sec:acceleration}
\subsection{Local conditions and finite entry}
\label{subsec:local-conditions}
Define the distance to the solution set by
\[
  r(x):=d(x,\Xstar)=\min_{y\in\Xstar}d(x,y).
\]

\begin{assumption}[Local H\"olderian error bound]
\label{ass:heb}
There exist $c>0$, $\theta\in(0,1/2]$, and $R_H>0$ such that
\begin{equation}
  r(x)\leq c\,h(x)^\theta,
  \qquad x\in\X\ \text{with}\ r(x)\leq R_H.
  \label{eq:heb}
\end{equation}
\end{assumption}

The Euclidean analysis of \citet{wirth2023acceleration} uses the normalized scaling inequality
\begin{equation}
 \frac{\langle\nabla f(x),x-p\rangle}{\|x-p\|}
 \geq
 \phi\frac{\langle\nabla f(x),x-x^\star\rangle}
               {\|x-x^\star\|},
 \label{eq:euclidean-scaling}
\end{equation}
where $p$ is an exact LMO solution.  In the strict-interior setting, the underlying geometric estimate goes back to \citet{guelat1986comments}.
For a nearest optimizer $x^\star$ and an admissible minimizing direction $v^\star$ from $x$ to $x^\star$, geodesic convexity gives $-\langle\grad f(x),v^\star\rangle_x\geq h(x)$.  This motivates comparing the oracle decrease per unit direction length with $h(x)/r(x)$.  Because exact oracle solutions need not have the same length, the condition below is stated uniformly over the oracle solution set.

\begin{definition}[Local length-normalized directional scaling]
\label{def:scaling}
Problem~\eqref{eq:problem} satisfies the $(\sigma,R_S)$ length-normalized directional scaling condition if there are $\sigma,R_S>0$ such that, whenever $x\notin\Xstar$ and $r(x)\leq R_S$,
\begin{equation}
  \frac{g(x)}{\bar\ell(x)}
  \geq \sigma\frac{h(x)}{r(x)}.
  \label{eq:scaling}
\end{equation}
\end{definition}

Because all exact oracle solutions have the same gap, $g(x)/\bar\ell(x)$ is their smallest gap-to-length ratio.
If $x\notin\Xstar$, Lemma~\ref{lem:gap-domination} gives $g(x)\geq h(x)>0$; hence every $v\in\mathcal S(x)$ is nonzero and $0<\norm{v}_x\leq\bar\ell(x)\leq D$.  So \eqref{eq:scaling} is equivalent to the corresponding inequality with $\bar\ell(x)$ replaced by $\norm{v}_x$ for every $v\in\mathcal S(x)$.

On the common local region $r(x)\leq\min\{R_H,R_S\}$, Assumption~\ref{ass:heb} and Definition~\ref{def:scaling} give, for every exact oracle solution,
\begin{equation}
 \frac{g(x)}{\ell(x)}
 \geq \frac{g(x)}{\bar\ell(x)}
 \geq \sigma\frac{h(x)}{r(x)}
 \geq \frac{\sigma}{c}h(x)^{1-\theta}.
 \label{eq:heb-scaling-chain}
\end{equation}
Inequality~\eqref{eq:heb-scaling-chain} links the objective to the selected oracle direction, whereas the set-scaling condition of \citet{scieur2026strongly} is a geometric property of the feasible set.

It remains to show that the local conditions used in \eqref{eq:heb-scaling-chain} eventually apply along the iterates.  The global baseline rate yields the following finite-entry property.

\begin{lemma}
\label{lem:entry}
Let $a\geq2$, suppose Assumptions~\ref{ass:regularity} and \ref{ass:heb} hold, and let $R_S>0$.  Define
\begin{equation}
 \bar R:=\min\{R_H,R_S\}.
 \label{eq:local-radius}
\end{equation}
If $\{x\in\X:r(x)\geq\bar R\}$ is nonempty, let
\begin{equation}
 \delta_{\bar R}:=
 \min_{\{x\in\X:r(x)\geq\bar R\}} h(x)>0.
 \label{eq:local-separation}
\end{equation}
Then any integer $S\geq1$ satisfying
\begin{equation}
 \frac{a^2LD^2}{2(S+a-1)}<\delta_{\bar R}
 \label{eq:local-entry-condition}
\end{equation}
has the property that $r(x_k)<\bar R$ for every generated iterate $k\geq S$.  If $\{x\in\X:r(x)\geq\bar R\}$ is empty, the conclusion holds with $S=1$.
\end{lemma}

\begin{proof}
Since $r$ is continuous and $\X$ is compact, the set $\{x\in\X:r(x)\geq\bar R\}$ is compact.  Because $\bar R>0$, it contains no optimizer. So $h$ is strictly positive on this set and $\delta_{\bar R}>0$.

For every $k\geq S$, Theorem~\ref{thm:baseline} and \eqref{eq:local-entry-condition} give
\[
 h_k
 \leq \frac{a^2LD^2}{2(k+a-1)}
 \leq \frac{a^2LD^2}{2(S+a-1)}
 <\delta_{\bar R}.
\]
By the definition of $\delta_{\bar R}$, this excludes $r(x_k)\geq\bar R$.  Hence $r(x_k)<\bar R$.
\end{proof}

The entry index $S$ in Lemma~\ref{lem:entry} depends on the separation constant $\delta_{\bar R}$ and is generally not explicit.  Combining the error bound with Theorem~\ref{thm:baseline} yields an explicit entry index.

\begin{corollary}[Explicit entry under a global error bound]
\label{cor:global-entry}
Let $a\geq2$, suppose Assumption~\ref{ass:regularity} holds, and assume that $r(x)\leq c h(x)^\theta$ for all $x\in\X$.
If $S\geq1$ satisfies
\begin{equation}
 c\left(\frac{a^2LD^2}{2(S+a-1)}\right)^\theta\leq R_S,
 \label{eq:global-entry-condition}
\end{equation}
then $r(x_k)\leq R_S$ for every generated iterate $k\geq S$.
In particular, one may take
\[
 S=
 \max\left\{
 1,\,
 \left\lceil
 \frac{a^2LD^2}{2}
 \left(\frac{c}{R_S}\right)^{1/\theta}
 -a+1
 \right\rceil
 \right\}.
\]
\end{corollary}

\subsection{Fast rate for the open-loop schedule}
\label{subsec:heb-open-loop}
The fast-rate analysis uses the scalar recurrence lemma in Appendix~\ref{app:sequence}, which gives an explicit bound without assuming that the suboptimality sequence is monotone.
\begin{theorem}[Fast open-loop rate]
\label{thm:abstract-acceleration}
Let $a>2$.  Suppose Assumptions~\ref{ass:regularity} and \ref{ass:heb} hold, and suppose problem~\eqref{eq:problem} satisfies the $(\sigma,R_S)$ scaling condition in Definition~\ref{def:scaling}.  Let $S\geq1$ be an entry index supplied by Lemma~\ref{lem:entry}, then the Algorithm~\ref{alg:ol-rfw} either terminates at an optimizer or
\begin{equation}
  h_k=O\left(k^{-1/(1-\theta)}\right).
  \label{eq:accelerated-rate}
\end{equation}
More precisely, for every $k\geq S$,
\begin{equation}
\begin{split}
 h_k\leq\max\Bigg\{
 \left(\frac{\eta_{k-2}}{\eta_{S-1}}\right)^{\!1/(1-\theta)}h_S,
 \left(\eta_{k-2}\frac{acLD}{2(a-2)\sigma}\right)^{\!1/(1-\theta)}
 +\frac{LD^2}{2}\eta_{k-2}^2
 \Bigg\},
\end{split}
\end{equation}
where $\eta_j=a/(j+a)$ for the auxiliary indices $j\geq-1$.
\end{theorem}

\begin{proof}
Lemma~\ref{lem:entry} places every $x_k$, $k\geq S$, in both local regions.  From \eqref{eq:heb-scaling-chain}, we have
\begin{equation}
  \frac{g_k}{\ell_k}
  \geq\frac{\sigma}{c}h_k^{1-\theta}.
  \label{eq:scaled-gap}
\end{equation}
By Lemma~\ref{lem:gap-domination}, a nonterminal iteration has $h_k>0$, $g_k>0$, and $\ell_k>0$.

Set $\lambda_a:=2/a\in(0,1)$ and split the linear term in \eqref{eq:progress} as $g_k=\lambda_a g_k+(1-\lambda_a)g_k$.  Use $g_k\geq h_k$ on the first part, \eqref{eq:scaled-gap} on the second, and $\ell_k^2\leq D\ell_k$ in the quadratic remainder.  This yields
\begin{equation*}
 h_{k+1}
 \leq(1-\lambda_a\eta_k)h_k
 -\eta_k(1-\lambda_a)\frac{\sigma}{c}\ell_kh_k^{1-\theta}
 +\eta_k^2\frac{LD}{2}\ell_k.
\end{equation*}
Since $\lambda_a a=2$, apply Lemma~\ref{lem:sequence} with
\[
 \lambda=\lambda_a,\qquad
 A=\frac{(a-2)\sigma}{ac},\qquad
 B=\frac{LD}{2},\qquad
 C_k=\ell_k,\qquad C=D.
\]
Its explicit bound is the displayed claim because $B/A=acLD/[2(a-2)\sigma]$ and $BC=LD^2/2$.  Finally, $\eta_{k-2}=a/(k+a-2)=O(1/k)$.
\end{proof}

\begin{corollary}[Quadratic open-loop rate]
\label{cor:theta-half}
Under the assumptions of Theorem~\ref{thm:abstract-acceleration}, if $\theta=1/2$, then
\[
 h_k\leq\eta_{k-2}^2
 \max\left\{
 \frac{h_S}{\eta_{S-1}^2},
 \left(\frac{acLD}{2(a-2)\sigma}\right)^2+\frac{LD^2}{2}
 \right\},
 \qquad k\geq S.
\]
Hence $h_k=O(1/k^2)$.
\end{corollary}

\subsection{Comparison with the short step}
\label{subsec:feedback-comparison}

For comparison, we analyze the standard Riemannian short step, obtained by minimizing the quadratic upper model in~\eqref{eq:generic-progress} with $\widehat L\geq L$; see~\citet{pokutta2024short,scieur2026strongly}.
The analysis uses the same local HEB and directional-scaling conditions as Theorem~\ref{thm:abstract-acceleration}.

\begin{proposition}[Rates for the  short step]
\label{prop:feedback}
Suppose Assumption~\ref{ass:regularity} holds, and $\widehat L\geq L$.  Use the oracle, stopping rule, and exponential-map update of Algorithm~\ref{alg:ol-rfw}, but replace its open-loop step size by
\begin{equation}
 \eta_k^{\rm fb}
 :=
 \min\left\{1,\frac{g_k}{\widehat L\ell_k^2}\right\}.
 \label{eq:feedback-step}
\end{equation}
Then, every nonterminal update is feasible and satisfies
\begin{equation}
 h_{k+1}\leq h_k-\min\left\{\frac{h_k}{2},
                 \frac{h_k^2}{2\widehat L D^2}\right\}.
 \label{eq:feedback-global}
\end{equation}
So the algorithm either terminates at an optimizer or $h_k\to0$.
And if Assumption~\ref{ass:heb} and the $(\sigma,R_S)$ scaling condition in Definition~\ref{def:scaling} also hold, set $C_{\rm fb}:=\sigma^2/(2\widehat Lc^2)$.  There is a finite, generally nonexplicit index $S_{\rm fb}$ such that, for every $k\geq S_{\rm fb}$,
\begin{equation}
 h_{k+1}\leq h_k-\min\left\{\frac{h_k}{2},
     C_{\rm fb}h_k^{2(1-\theta)}\right\}.
 \label{eq:feedback-local}
\end{equation}
Moreover,
\begin{align}
 h_k&=O\left(k^{-1/(1-2\theta)}\right),
 &&0<\theta<\tfrac12,
 \label{eq:feedback-polynomial}\\
 h_k&=O\left((1-\min\{1/2,C_{\rm fb}\})^k\right),
 &&\theta=\tfrac12.
 \label{eq:feedback-linear}
\end{align}
\end{proposition}

\begin{proof}
Since $\eta_k^{\rm fb}\in(0,1]$ and $v_k\in\Dset(x_k)$, every update remains in $\X$.  If \eqref{eq:feedback-step} selects the full step, then $g_k\geq\widehat L\ell_k^2$, and \eqref{eq:generic-progress} with $L\leq\widehat L$ gives
\[
 h_{k+1}\leq h_k-g_k+\frac{\widehat L}{2}\ell_k^2
 \leq h_k-\frac{g_k}{2}\leq h_k-\frac{h_k}{2}.
\]
For a short step, the same model gives
\[
 h_{k+1}\leq h_k-\frac{g_k^2}{2\widehat L\ell_k^2}
 \leq h_k-\frac{h_k^2}{2\widehat L D^2}.
\]
Then we obtain \eqref{eq:feedback-global}.  If the method does not terminate, $\{h_k\}$ is nonincreasing and converges to some $\bar h\geq0$.
If $\bar h>0$, then every iteration decreases $h_k$ by at least
$
\min\left\{\frac{\bar h}{2}, \frac{\bar h^2}{2\widehat L D^2}\right\}>0,
$
which contradicts $h_k\geq0$.  So $h_k\to0$.

Now set $\bar R:=\min\{R_H,R_S\}$.
If $E_{\bar R}:=\{x\in\X:r(x)\geq\bar R\}$ is nonempty, compactness and
$E_{\bar R}\cap\Xstar=\varnothing$ give
$\delta_{\bar R}:=\min_{x\in E_{\bar R}}h(x)>0$.  Since $h_k\to0$, all
sufficiently late iterates satisfy $h_k<\delta_{\bar R}$ and hence
$r(x_k)<\bar R$. The same conclusion is immediate if $E_{\bar R}$ is empty.
We define a finite index $S_{\rm fb}$.
For every $k\geq S_{\rm fb}$,
\[
 \frac{g_k}{\ell_k}
 \geq\frac{g_k}{\bar\ell(x_k)}
 \geq\frac{\sigma}{c}h_k^{1-\theta}.
\]
In the full-step case, the estimate above gives
\[
 h_{k+1}\leq h_k-\frac{h_k}{2}.
\]
In the short-step case,
\[
 \frac{g_k^2}{2\widehat L\ell_k^2}
 =\frac{1}{2\widehat L}
   \left(\frac{g_k}{\ell_k}\right)^2
 \geq
 \frac{\sigma^2}{2\widehat Lc^2}
 h_k^{2(1-\theta)}
 =C_{\rm fb}h_k^{2(1-\theta)}.
\]
Combining the two cases proves \eqref{eq:feedback-local}.

Lemma~\ref{lem:gap-domination} gives $h_k>0$ at every nonterminal iteration.
For $0<\theta<1/2$, set $q:=1-2\theta>0$.  Since $h_k\to0$, choose $K\geq S_{\rm fb}$ such that $z_k:=C_{\rm fb}h_k^q\leq1/2$ for every $k\geq K$.
Now the second term in the minimum in \eqref{eq:feedback-local} is the smaller one, so
\[
 h_{k+1}\leq h_k(1-z_k).
\]
By convexity, $(1-z)^{-q}\geq1+qz$ for $z\in[0,1)$, Consequently,
\[
 h_{k+1}^{-q}
 \geq h_k^{-q}(1-z_k)^{-q}
 \geq h_k^{-q}+qC_{\rm fb}.
\]
Summing gives $h_k^{-q}\geq h_K^{-q}+qC_{\rm fb}(k-K)$, namely, $ h_k \le \left( h_K^{-q}+qC_{\rm fb}(k-K) \right)^{-1/q} $, which proves \eqref{eq:feedback-polynomial}.
If $\theta=1/2$, then \eqref{eq:feedback-local} gives, for every $k\geq S_{\rm fb}$,
\[
 h_{k+1}\leq
 \left(1-\min\{1/2,C_{\rm fb}\}\right)h_k,
\]
which proves \eqref{eq:feedback-linear}.
\end{proof}

\section{Geometric conditions for accelerated rates}
\label{sec:geometry}
This section considers two geometric conditions that lead to faster RFW convergence. The first combines a Riemannian set-scaling inequality with a uniform positive lower bound on the gradient norm, with spherical balls as an explicit example. The second derives length-normalized directional scaling from an interior-point condition and applies the result to relative interiors and SPD Loewner intervals.
\subsection{Riemannian set scaling and strong gap growth}
\label{subsec:rsi}
Unlike Definition~\ref{def:scaling}, the following condition depends only on the geometry of $\X$.  For $x\in\X$ and $w\in T_x\M$, define
\begin{equation}
 \mathcal A(x,w):=
 \arg\max_{v\in\Dset(x)}\inner{w}{v}_x.
 \label{eq:support-directions}
\end{equation}

\begin{definition}
\label{def:rsi}
For $\alpha>0$, the set $\X$ satisfies the $\alpha$-Riemannian scaling inequality (RSI) if
\begin{equation}
 \inner{w}{v}_x
 \geq \alpha\norm{w}_x\norm{v}_x^2
 \quad
 \text{for all }x\in\X,\ w\in T_x\M,\
 v\in\mathcal A(x,w).
 \label{eq:rsi}
\end{equation}
\end{definition}
On a Hadamard manifold, Definition~\ref{def:rsi} reduces to the Riemannian scaling inequality of \citet{scieur2026strongly}.  Its admissible-direction form remains meaningful when minimizing geodesics are nonunique.  The next proposition combines this condition with a uniform positive lower bound on $\norm{\grad f(x)}_x$. Such a bound is automatic when the gradient is continuous and nowhere zero on the compact set $\X$.

\begin{proposition}
\label{prop:rsi-growth}
Suppose Assumption~\ref{ass:regularity} holds, $\X$ satisfies Definition~\ref{def:rsi}, and
\begin{equation}
 \nu:=\inf_{x\in\X}\norm{\grad f(x)}_x>0.
 \label{eq:gradient-away-zero}
\end{equation}
Let
\begin{equation}
 M_{\rm RSI}:=\frac{L}{\alpha\nu},
 \label{eq:rsi-growth-constant}
\end{equation}
then for every $x\in\X$, every exact oracle solution $v\in\mathcal S(x)$, and every $\eta\in[0,1]$,
\begin{align}
 g(x)&\geq\alpha\nu\norm{v}_x^2,
 \label{eq:rsi-gap-length}\\
 f(\Exp_x(\eta v))-f(x)
 -\eta\inner{\grad f(x)}{v}_x
 &\leq\frac{M_{\rm RSI}}{2}\eta^2g(x).
 \label{eq:rsi-gap-remainder}
\end{align}
\end{proposition}

\begin{proof}
For $w=-\grad f(x)$, the maximizing set $\mathcal A(x,w)$ coincides with the exact oracle solution set $\mathcal S(x)$.  Definition~\ref{def:rsi} and \eqref{eq:gradient-away-zero} therefore give
\[
 g(x)=\inner{-\grad f(x)}{v}_x
 \geq\alpha\norm{\grad f(x)}_x\norm{v}_x^2
 \geq\alpha\nu\norm{v}_x^2,
\]
which proves \eqref{eq:rsi-gap-length}.
Combined with \eqref{eq:gsmooth}, we have
\[
 f(\Exp_x(\eta v))-f(x)
 -\eta\inner{\grad f(x)}{v}_x
 \leq\frac{L}{2}\eta^2\norm{v}_x^2
 \leq\frac{L}{2\alpha\nu}\eta^2g(x)
 =\frac{M_{\rm RSI}}{2}\eta^2g(x).
\]
\end{proof}

Condition~\eqref{eq:gradient-away-zero} excludes stationary points in $\X$. In particular, an optimizer cannot lie in the ambient interior of $\X$.  Indeed, the interior-ball condition in Assumption~\ref{ass:interior} implies $\grad f(x^\star)=0$.  Thus Theorem~\ref{thm:rsi-open-loop} and the interior-based results below cover complementary optimization regimes.

\begin{theorem}[Open-loop rate under Riemannian set scaling]
\label{thm:rsi-open-loop}
Let $a\geq2$.  Suppose Assumption~\ref{ass:regularity} holds, $\X$ satisfies the $\alpha$-Riemannian scaling inequality, and \eqref{eq:gradient-away-zero} holds.  Let $M_{\rm RSI}$ be defined by \eqref{eq:rsi-growth-constant}, and define the threshold index
\begin{equation}
 K_{\rm RSI}:=
 \max\left\{1,
 \left\lceil\frac{aM_{\rm RSI}}{2}-a\right\rceil
 \right\}.
 \label{eq:rsi-threshold}
\end{equation}
Algorithm~\ref{alg:ol-rfw} either terminates at an optimizer or, for every $k\geq K_{\rm RSI}$,
\begin{align}
 h_k
 &\leq h_{K_{\rm RSI}}
 \exp\left(
   \frac{M_{\rm RSI}a^2}
        {2(K_{\rm RSI}+a-1)}
 \right)
 \left(\frac{K_{\rm RSI}+a}{k+a}\right)^a
 \label{eq:rsi-rate-hK}\\
 &\leq
 \frac{a^2LD^2}{2(K_{\rm RSI}+a-1)}
 \exp\left(
   \frac{La^2}
        {2\alpha\nu(K_{\rm RSI}+a-1)}
 \right)
 \left(\frac{K_{\rm RSI}+a}{k+a}\right)^a.
 \label{eq:rsi-rate-explicit}
\end{align}
In particular, $h_k=O(k^{-a})$.
\end{theorem}

\begin{proof}
Suppose that it does not terminate.
For every $j\geq K_{\rm RSI}$, ~\eqref{eq:rsi-threshold} gives
\[
 j+a\geq \frac{aM_{\rm RSI}}{2}.
\]
Since $\eta_j=a/(j+a)$, it follows that
\[
 \frac{M_{\rm RSI}}{2}\eta_j
 =\frac{aM_{\rm RSI}}{2(j+a)}
 \leq1.
\]

Lemma~\ref{lem:progress} and Proposition~\ref{prop:rsi-growth} imply
\begin{equation}
 h_{j+1}
 \leq h_j-\eta_j
 \left(1-\frac{M_{\rm RSI}}{2}\eta_j\right)g_j.
 \label{eq:rsi-one-step}
\end{equation}
The coefficient of $g_j$ in \eqref{eq:rsi-one-step} is then nonnegative.  Using $g_j\geq h_j$ therefore yields
\begin{equation}
 h_{j+1}
 \leq
 \left(1-\eta_j+\frac{M_{\rm RSI}}{2}\eta_j^2\right)h_j,
 \qquad j\geq K_{\rm RSI}.
 \label{eq:rsi-linear-recurrence}
\end{equation}
Iterating \eqref{eq:rsi-linear-recurrence} and using $1+u\leq e^u$ gives, for $k\geq K_{\rm RSI}$,
\begin{align*}
 h_k
 &\leq h_{K_{\rm RSI}}
 \exp\left\{
 -a\sum_{j=K_{\rm RSI}}^{k-1}\frac{1}{j+a}
 +\frac{M_{\rm RSI}a^2}{2}
  \sum_{j=K_{\rm RSI}}^{k-1}\frac{1}{(j+a)^2}
 \right\}.
\end{align*}
With the fact that
\[
 \sum_{j=K_{\rm RSI}}^{k-1}\frac{1}{j+a}
 \geq\log\left(\frac{k+a}{K_{\rm RSI}+a}\right),
 \qquad
 \sum_{j=K_{\rm RSI}}^{k-1}\frac{1}{(j+a)^2}
 \leq\frac{1}{K_{\rm RSI}+a-1},
\]
\eqref{eq:rsi-rate-hK} holds.
Since $K_{\rm RSI}\geq1$, the global baseline bound in Theorem~\ref{thm:baseline} applies at $K_{\rm RSI}$.
Substituting that bound and $M_{\rm RSI}=L/(\alpha\nu)$ proves \eqref{eq:rsi-rate-explicit}.
\end{proof}

The same rate holds for the minimum Frank--Wolfe gap over the latter half of the iterates.

\begin{corollary}
\label{cor:rsi-gap-rate}
Under the assumptions of Theorem~\ref{thm:rsi-open-loop}, define
\begin{equation}
 \widehat K_{\rm RSI}:=
 \max\left\{K_{\rm RSI},
 \left\lceil aM_{\rm RSI}-a\right\rceil\right\}.
 \label{eq:rsi-gap-threshold}
\end{equation}
For $k\geq2$, set
\begin{equation}
 g_k^{\rm rec}:=
 \min_{\lfloor k/2\rfloor\leq j<k}g_j.
 \label{eq:recent-gap}
\end{equation}
If the algorithm does not terminate, then, for all sufficiently large $k$,
\begin{equation}
 g_k^{\rm rec}=O(k^{-a}).
 \label{eq:rsi-gap-rate}
\end{equation}
More precisely, if $m:=\lfloor k/2\rfloor\geq\widehat K_{\rm RSI}$ and $k\geq\lceil2a\rceil$, then
\begin{equation}
 g_k^{\rm rec}
 \leq\frac{2h_m}{a\log(3/2)}.
 \label{eq:rsi-gap-explicit}
\end{equation}
\end{corollary}

\begin{proof}
For $j\geq\widehat K_{\rm RSI}$, one has $M_{\rm RSI}\eta_j\leq1$.  Hence \eqref{eq:rsi-one-step} gives
\[
 h_{j+1}\leq h_j-\frac{\eta_j}{2}g_j.
\]
Summing from $m$ to $k-1$ and using $h_k\geq0$ yields
\[
 \frac12 g_k^{\rm rec}\sum_{j=m}^{k-1}\eta_j
 \leq\frac12\sum_{j=m}^{k-1}\eta_jg_j\leq h_m.
\]
Furthermore,
\[
 \sum_{j=m}^{k-1}\eta_j
 \geq a\log\left(\frac{k+a}{m+a}\right)
 \geq a\log(3/2),
\]
where the last inequality follows from $m\leq k/2$ and $k\geq2a$.
This proves \eqref{eq:rsi-gap-explicit}.  Theorem~\ref{thm:rsi-open-loop} gives $h_m=O(m^{-a})=O(k^{-a})$, proving \eqref{eq:rsi-gap-rate}.
\end{proof}

Proposition~\ref{prop:rsi-growth} shows that exact Riemannian scaling, together with the gradient lower bound, yields the Riemannian analogue of strong $(M,1)$-growth in~\cite{wirth2025affine}.
The explicit bound also shows a trade-off in the schedule parameter.
Larger values of $a$ improve the asymptotic exponent but may increase the threshold and prefactor.
\subsubsection{Spherical balls}

As a non-Hadamard ambient manifold, consider the unit round sphere
\[
 \mathbb S^{n-1}
 :=
 \{x\in\mathbb R^n:\|x\|_2=1\},
 \qquad n\geq2,
\]
with its standard metric and geodesic distance \(d_{\mathbb S}\).  For \(o\in\mathbb S^{n-1}\) and \(R>0\), define the closed geodesic ball
\[
 \overline B_{\mathbb S}(o,R)
 :=
 \{x\in\mathbb S^{n-1}:d_{\mathbb S}(x,o)\leq R\}.
\]
The next proposition establishes an exact Riemannian scaling inequality for these balls when \(R<\pi/2\).

\begin{proposition}[Riemannian scaling inequality on a spherical ball]
\label{prop:sphere-ball-scaling}
Let $n\geq2$, let $o\in\mathbb S^{n-1}$,  $0<R<\pi/2$ and set $\X:=\overline B_{\mathbb S}(o,R)$.
Then $\X$ is compact and contains the unique minimizing geodesic joining each pair of its points.  
Moreover, $\X$ satisfies Definition~\ref{def:rsi} with
$\alpha_R=\frac12\cot R$; that is, for every
$x\in\X$, $w\in T_x\mathbb S^{n-1}$, and
$v\in\mathcal A(x,w)$,
\begin{equation}
 \inner{w}{v}_x
 \geq \frac12\cot R\,\norm{w}_x\norm{v}_x^2.
 \label{eq:sphere-scaling}
\end{equation}

\end{proposition}

\begin{proof}
The set $\X$ is closed in the compact sphere.  Let $p,q\in\X$ and set $\rho=d(p,q)$.  If $p\neq q$, then $\rho\leq2R<\pi$, so their minimizing geodesic is unique and has the representation
\[
 \gamma(t)=
 \frac{\sin((1-t)\rho)}{\sin\rho}p
 +\frac{\sin(t\rho)}{\sin\rho}q,
 \qquad 0\leq t\leq1.
\]
The coefficients are nonnegative, and
\[
 \frac{\sin((1-t)\rho)+\sin(t\rho)}{\sin\rho}
 =\frac{\cos((1-2t)\rho/2)}{\cos(\rho/2)}
 \geq1.
\]
Since $o^\top p,o^\top q\geq\cos R>0$, we obtain $o^\top\gamma(t)\geq\cos R$, and hence $\gamma(t)\in\X$.

Fix $x\in\X$, $w\in T_x\mathbb S^{n-1}$, and $v\in\mathcal A(x,w)$.  The result is immediate if $w=0$ or $v=0$.
Otherwise, set
\[
 z:=\Exp_x(v),\qquad \tau:=\norm{v}_x=d(x,z).
\]
Then $0<\tau\leq2R<\pi$.  The exponential map is
\[
 \Exp_x(v)=\cos(\tau)x+\frac{\sin(\tau)}{\tau}v;
\]
see \citet[Examples~5.37]{boumal2023introduction}.
If $\xi=\xi_{\parallel}+\xi_{\perp}$ is decomposed into components parallel and orthogonal to $v$, differentiation gives
\[
 \norm{D\Exp_x(v)[\xi]}_z^2
 =\norm{\xi_{\parallel}}_x^2
 +\left(\frac{\sin\tau}{\tau}\right)^2
  \norm{\xi_{\perp}}_x^2.
\]
Thus $D\Exp_x(v)$ is nonsingular and
$
\norm{D\Exp_x(v)}_{\rm op}\leq1,
$
where $\|\cdot\|_{\rm op}$ denotes the operator norm induced by the Riemannian norms on $T_x\mathbb S^{n-1}$ and $T_z\mathbb S^{n-1}$.

Define
\[
 q(y):=\frac12d(y,o)^2,\qquad
 \psi_x(\xi):=q(\Exp_x(\xi)),\qquad
 K_x:=\Log_x(\X).
\]


For every $y\in\X$, the triangle inequality gives
$d(x,y)\leq2R<\pi=\operatorname{inj}(x)$.  Hence $\Log_x(y)$ is uniquely
defined and represents the unique minimizing geodesic from $x$ to $y$
\citep[Example~10.21 and Proposition~10.22]{boumal2023introduction}.
Together with the geodesic convexity proved above, this gives
\[
 \Dset(x)=K_x
 =\left\{\xi\in T_x\mathbb S^{n-1}:
 \norm{\xi}_x<\pi,\
 \psi_x(\xi)\leq\frac{R^2}{2}\right\}.
\]

If $z\in\operatorname{int}\X$, local invertibility of $\Exp_x$ would make
$v$ an interior point of $K_x$, contradicting the optimality of $v$ because
$w\neq0$.  Thus $z\in\partial\X$ and $q(z)=R^2/2$.



Since $D\Exp_x(v)$ is nonsingular and $\grad q(z)\neq0$,
\[
 \grad\psi_x(v)=D\Exp_x(v)^*\grad q(z)\neq0.
\]
Thus the active constraint satisfies the standard constraint qualification,
and the first-order necessary condition gives
\begin{equation}
 w=\lambda\grad\psi_x(v)
 \quad\text{for some }\lambda>0.
 \label{eq:sphere-kkt}
\end{equation}


Let $r(y):=d(y,o)$.
For $0<r(y)\leq R$, taking $\Delta=1$ in the lower Hessian comparison
bound of \citet[eqs.~(2.3)--(2.5)]{afsari2013convergence}, and using
quadratic homogeneity in $\eta$, gives
\[
 \operatorname{Hess}q(y)[\eta,\eta]
 \geq r(y)\cot r(y)\,\norm{\eta}_y^2,
 \qquad \eta\in T_y\mathbb S^{n-1}.
\]
At $y=o$,
$\operatorname{Hess}q(o)[\eta,\eta]=\norm{\eta}_o^2$.
Since $r\mapsto r\cot r$ decreases on $(0,\pi/2)$ and
$R\cot R\leq1$, it follows that
\[
 \operatorname{Hess}q\succeq R\cot R\,g
 \quad\text{on }\X.
\]


Let $\phi(t):=q(\Exp_x(tv))$.  By geodesic convexity,
$\Exp_x(tv)\in\X$, and hence
\[
 \phi''(t)\geq R\cot R\,\norm{v}_x^2.
\]
Using
\[
 \phi(0)=\phi(1)-\phi'(1)+\int_0^1t\phi''(t)\,dt
\]
together with $q(x)\leq q(z)$ gives
\[
 \inner{\grad\psi_x(v)}{v}_x
 \geq\frac{R\cot R}{2}\norm{v}_x^2.
\]
Moreover,
\[
 \norm{\grad\psi_x(v)}_x
 \leq\norm{D\Exp_x(v)}_{\rm op}\norm{\grad q(z)}_z
 \leq R.
\]
Since $w=\lambda\grad\psi_x(v)$, we have
$\lambda\geq\norm{w}_x/R$, and therefore
\[
 \inner{w}{v}_x
 =\lambda\inner{\grad\psi_x(v)}{v}_x
 \geq\lambda\frac{R\cot R}{2}\norm{v}_x^2
 \geq\frac12\cot R\,\norm{w}_x\norm{v}_x^2.
\]
\end{proof}

\begin{corollary}[Open-loop rates on spherical balls]
\label{cor:sphere-general-rate}
Let $n\geq2$, let $o\in\mathbb S^{n-1}$, and let $0<R<\pi/2$.  Set $\X:=\overline B_{\mathbb S}(o,R)$.  Suppose Assumption~\ref{ass:regularity} holds with constant $L$ and let $\nu:=\inf_{x\in\X}\norm{\grad f(x)}_x>0$.
Fix $a\geq2$ and define
\begin{equation}
 \begin{aligned}
 M_R&:=\frac{2L\tan R}{\nu},\\
 K_R&:=\max\left\{1,
 \left\lceil\frac{aL\tan R}{\nu}-a\right\rceil\right\},\\
 \widehat K_R&:=\max\left\{K_R,
 \left\lceil\frac{2aL\tan R}{\nu}-a\right\rceil\right\}.
 \end{aligned}
 \label{eq:sphere-general-constants}
\end{equation}
Algorithm~\ref{alg:ol-rfw} either terminates at an optimizer or generates an infinite sequence.  In the latter case, for every $k\geq K_R$,
\begin{equation}
 h_k\leq
 \frac{2a^2LR^2}{K_R+a-1}
 \exp\left(
  \frac{a^2L\tan R}{\nu(K_R+a-1)}
 \right)
 \left(\frac{K_R+a}{k+a}\right)^a.
 \label{eq:sphere-general-rate}
\end{equation}
Moreover, if $m:=\lfloor k/2\rfloor\geq\widehat K_R$ and $k\geq\lceil2a\rceil$, then
\begin{equation}
 g_k^{\rm rec}\leq\frac{2h_m}{a\log(3/2)}.
 \label{eq:sphere-general-gap}
\end{equation}
Consequently, $h_k=O(k^{-a})$ and $g_k^{\rm rec}=O(k^{-a})$.
\end{corollary}

\begin{proof}
Proposition~\ref{prop:sphere-ball-scaling} gives $\alpha_R=\frac12\cot R$.  The diameter of $\X$ is $D=2R$.  Indeed, the triangle inequality gives $D\leq2R$, while, for any unit vector $u\in T_o\mathbb S^{n-1}$, the points $\Exp_o(Ru)$ and $\Exp_o(-Ru)$ belong to $\X$ and are at distance $2R$ because $2R<\pi$.
Thus $M_{\rm RSI}=M_R$, $K_{\rm RSI}=K_R$, and $\widehat K_{\rm RSI}=\widehat K_R$.  Substituting these identities and $D=2R$ into Theorem~\ref{thm:rsi-open-loop} and Corollary~\ref{cor:rsi-gap-rate} gives \eqref{eq:sphere-general-rate} and \eqref{eq:sphere-general-gap}.
\end{proof}

\begin{corollary}[An exterior-target objective on a spherical ball]
\label{cor:sphere-ball}
Let $n\geq2$ and $o,y\in\mathbb S^{n-1}$.  Let $R,\Delta>0$ satisfy $d(o,y)=R+\Delta$ and $2R+\Delta<\pi/2$, and set $\X=\overline B_{\mathbb S}(o,R)$ and $f(x)=1-x^\top y$.  Then Assumption~\ref{ass:regularity} holds with $L=1$, and $f$ is $\cos(2R+\Delta)$-strongly geodesically convex on $\X$.  Furthermore,
\begin{equation}
 \alpha_R=\frac12\cot R,
 \qquad
 \nu:=\inf_{x\in\X}\norm{\grad f(x)}_x=\sin\Delta.
 \label{eq:sphere-problem-constants}
\end{equation}
The unique minimizer lies on $\partial\X$ and is given by $x^\star=\Exp_o(Ru_o)$, where $u_o=\Log_o(y)/(R+\Delta)$.  For the constants in Corollary~\ref{cor:sphere-general-rate}, one may take
\begin{equation}
 M_R=\frac{2\tan R}{\sin\Delta},
 \qquad
 K_R=
 \max\left\{1,
 \left\lceil\frac{a\tan R}{\sin\Delta}-a\right\rceil\right\}.
 \label{eq:sphere-rate-constants}
\end{equation}
Consequently, for every $a\geq2$, RFW-OL($a$) either terminates at $x^\star$ or satisfies $h_k=O(k^{-a})$ and $g_k^{\rm rec}=O(k^{-a})$.
\end{corollary}

\begin{proof}
Since $R<\pi/2$, Proposition~\ref{prop:sphere-ball-scaling} gives $\alpha_R=\frac12\cot R$.  For every $x\in\X$, the triangle inequality gives
\[
 \Delta\leq d(x,y)\leq2R+\Delta<\frac{\pi}{2}.
\]
On the unit sphere,
\[
 \grad f(x)=(x^\top y)x-y,
 \qquad
 \operatorname{Hess}f(x)[\xi,\xi]
 =(x^\top y)\norm{\xi}_x^2.
\]
Since $x^\top y=\cos d(x,y)$, for every $\xi\in T_x\mathbb S^{n-1}$,
\[
 \cos(2R+\Delta)\norm{\xi}_x^2
 \leq \operatorname{Hess}f(x)[\xi,\xi]
 \leq \norm{\xi}_x^2.
\]
Integrating these bounds along the minimizing geodesics in $\X$ proves Assumption~\ref{ass:regularity} with $L=1$ and the stated strong-convexity constant.

Because $\norm{\Log_o(y)}_o=R+\Delta$, the vector $u_o$ has unit norm.
The points $o$, $x^\star=\Exp_o(Ru_o)$, and $y=\Exp_o((R+\Delta)u_o)$ lie on the same minimizing geodesic, so
\[
 d(o,x^\star)=R,
 \qquad
 d(x^\star,y)=\Delta.
\]
Hence $x^\star\in\partial\X$ attains the lower bound $d(x,y)\geq\Delta$ and minimizes $f$; strong geodesic convexity gives uniqueness.  Moreover,
\[
 \norm{\grad f(x)}_x=\sin d(x,y)\geq\sin\Delta,
\]
where the inequality uses $d(x,y)<\pi/2$, and equality holds at $x^\star$.  This proves the second identity in \eqref{eq:sphere-problem-constants}.  Substituting $L=1$ and $\nu=\sin\Delta$ into \eqref{eq:sphere-general-constants} gives \eqref{eq:sphere-rate-constants}.  The rate statements follow from Corollary~\ref{cor:sphere-general-rate}.
\end{proof}

\subsection{Directional scaling from interiority}
\label{subsec:interior-scaling}
On a complete manifold, the interior-scaling argument below requires a uniform lower bound on the injectivity radius.  For $x\in\M$, let $\inj(x)$ denote the supremum of the radii on which $\Exp_x$ is a diffeomorphism from the corresponding tangent-space ball onto its image \citep[Section~10.2]{boumal2023introduction}.  The injectivity-radius function is positive and continuous.  Since $\X$ is compact,
\[
  \iota_{\X}:=\min_{x\in\X}\inj(x)>0
\]
is therefore well defined~\citep{sak1996riemannian}.
If $\M$ is Hadamard, then $\inj(x)=+\infty$ for every $x$, and hence $\iota_{\X}=+\infty$.

For $x\in\M$ and $r>0$, denote the closed Riemannian ball by
\[
 \overline B_{\mathcal M}(x,r)
 :=
 \{y\in\M:d(x,y)\leq r\}.
\]
\begin{assumption}[Unique strict-interior optimizer]
\label{ass:interior}
Problem~\eqref{eq:problem} has a unique optimizer $x^\star$, and there exists $\beta>0$ such that
\begin{equation}
  \overline B_{\mathcal M}(x^\star,\beta)\subseteq\X.
  \label{eq:interior-ball}
\end{equation}
\end{assumption}

The ball in Assumption~\ref{ass:interior} is taken in the ambient manifold $\M$, so the condition does not cover lower-dimensional feasible sets.

\begin{proposition}[Interiority implies length-normalized directional scaling]
\label{prop:interior-scaling}
Suppose Assumptions~\ref{ass:regularity} and~\ref{ass:interior} hold.  Set
\begin{equation}
 \rho:=\min\left\{\frac{\beta}{2},\frac{\iota_{\X}}{2}\right\},
 \qquad
 \sigma:=\frac{\rho}{D}.
 \label{eq:interior-constants}
\end{equation}
Then problem~\eqref{eq:problem} satisfies the $(\sigma,\beta/2)$ scaling condition in Definition~\ref{def:scaling}.  If $\M$ is Hadamard, one may take $\rho=\beta/2$ and $\sigma=\beta/(2D)$.
\end{proposition}

\begin{proof}
Fix $x\neq x^\star$ with $d(x,x^\star)\leq\beta/2$.  Since $x^\star$ is the unique optimizer, $h(x)>0$ and $r(x)=d(x,x^\star)$.  Choose $v^\star\in\Dset(x)$ such that $\Exp_x(v^\star)=x^\star$.  First-order geodesic convexity and Cauchy--Schwarz give
\[
 h(x)
 \leq-\inner{\grad f(x)}{v^\star}_x
 \leq\norm{\grad f(x)}_x r(x).
\]
In particular, $\grad f(x)\neq0$.

Define
\[
 w:=-\rho\frac{\grad f(x)}{\norm{\grad f(x)}_x}.
\]
Since $\norm{w}_x=\rho<\inj(x)$, the geodesic $t\mapsto\Exp_x(tw)$ is minimizing.  For every $t\in[0,1]$,
\[
 d(x^\star,\Exp_x(tw))
 \leq d(x^\star,x)+d(x,\Exp_x(tw))
 =d(x^\star,x)+t\rho
 \leq\beta.
\]
Thus $\Exp_x(tw)\in\X$ by \eqref{eq:interior-ball}, and hence $w\in\Dset(x)$.  Oracle optimality now yields
\[
 g(x)\geq-\inner{\grad f(x)}{w}_x
 =\rho\norm{\grad f(x)}_x>0.
\]
Consequently, $\bar\ell(x)>0$, and
\[
 \frac{g(x)}{\bar\ell(x)}
 \geq\frac{\rho}{D}\norm{\grad f(x)}_x
 \geq\sigma\frac{h(x)}{r(x)}.
\]
This is the $(\sigma,\beta/2)$ scaling condition.
\end{proof}

\begin{corollary}[Fast open-loop rate under interiority]
\label{cor:interior-rate}
Let $a>2$, and suppose Assumptions~\ref{ass:regularity}, \ref{ass:heb}, and~\ref{ass:interior} hold.  With $\rho$ and $\sigma$ defined by \eqref{eq:interior-constants}, Algorithm~\ref{alg:ol-rfw} either terminates at $x^\star$ or satisfies
\[
 f(x_k)-f^\star
 =O\left(k^{-1/(1-\theta)}\right).
\]
The scaling radius and constant are $R_S=\beta/2$ and $\sigma=\rho/D$, respectively.
If the H\"olderian error bound holds globally, it may be chosen according to \eqref{eq:global-entry-condition} with $R_S=\beta/2$.
\end{corollary}

\begin{proof}
Proposition~\ref{prop:interior-scaling} verifies the required scaling condition.
The entry index is finite by Lemma~\ref{lem:entry}.
The result follows from Theorem~\ref{thm:abstract-acceleration} and, in the global case, Corollary~\ref{cor:global-entry}.
\end{proof}

\begin{corollary}[Quadratic open-loop rate under strong geodesic convexity]
\label{cor:strong-convexity}
Suppose Assumptions~\ref{ass:regularity} and~\ref{ass:interior} hold, and suppose that $f$ is $\mu$-strongly geodesically convex in the sense of Definition~\ref{def:strong-gconvexity}.  Then, for every $x\in\X$,
\[
 d(x,x^\star)\leq\sqrt{\frac{2h(x)}{\mu}}.
\]
Thus the H\"olderian error bound holds globally with $c=\sqrt{2/\mu}$ and $\theta=1/2$.  Consequently, for every $a>2$, Algorithm~\ref{alg:ol-rfw} either terminates at $x^\star$ or satisfies
\[
 f(x_k)-f^\star=O(1/k^2).
\]
\end{corollary}

\begin{proof}
Assumption~\ref{ass:interior} makes $x^\star$ an unconstrained local minimizer in $\M$, so $\grad f(x^\star)=0$.  For any $x\in\X$, choose $v\in\Dset(x^\star)$ such that $\Exp_{x^\star}(v)=x$.  Then \eqref{eq:first-order-strong-convexity} gives
\[
 h(x)\geq\frac{\mu}{2}\norm{v}_{x^\star}^2
       =\frac{\mu}{2}d(x,x^\star)^2.
\]
The stated error bound follows, and Corollary~\ref{cor:interior-rate} gives the rate.
\end{proof}

\subsubsection{Relative-interior extension}

The interior argument also applies relative to a prescribed totally geodesic submanifold.

\begin{corollary}
\label{cor:relative-interior}
Let $\mathcal N\subseteq\M$ be a connected, complete, embedded, totally geodesic submanifold containing $\X$, endowed with the induced metric.
Throughout this corollary, all geometric quantities are computed intrinsically on $\mathcal N$.  Suppose that $\X$ is weakly minimizing-geodesically convex in $\mathcal N$ and that $f|_{\mathcal N}$ satisfies Assumption~\ref{ass:regularity} with smoothness constant $L_{\mathcal N}$.

Assume that the optimizer $x^\star$ is unique and that, for some $\beta>0$,
\[
 \overline B_{\mathcal N}(x^\star,\beta)\subseteq\X.
\]
Set
\[
 D_{\mathcal N}:=\diam_{\mathcal N}(\X),
 \qquad
 \iota_{\X}^{\mathcal N}
 :=\min_{x\in\X}\inj_{\mathcal N}(x)>0,
\]
and suppose that $D_{\mathcal N}>0$.  Then the intrinsic length-normalized directional-scaling condition holds with
\[
 R_S=\frac{\beta}{2},
 \qquad
 \rho_{\mathcal N}
 =\min\left\{\frac{\beta}{2},
              \frac{\iota_{\X}^{\mathcal N}}{2}\right\},
 \qquad
 \sigma_{\mathcal N}
 =\frac{\rho_{\mathcal N}}{D_{\mathcal N}}.
\]

If Assumption~\ref{ass:heb} holds with $\theta\in(0,1/2]$, then, for every $a>2$, Algorithm~\ref{alg:ol-rfw} either terminates at $x^\star$ or satisfies
\[
 f(x_k)-f^\star
 =O\left(k^{-1/(1-\theta)}\right).
\]
If $f|_{\mathcal N}$ is $\mu$-strongly geodesically convex in $\mathcal N$, then the intrinsic error bound holds globally with $c=\sqrt{2/\mu}$ and $\theta=1/2$, and hence
\[
 f(x_k)-f^\star=O(1/k^2).
\]
\end{corollary}

\begin{proof}
Completeness and compactness provide the standing geometric properties on $\mathcal N$.  Proposition~\ref{prop:interior-scaling} gives the stated scaling constant.  The two rates follow from Corollaries~\ref{cor:interior-rate} and~\ref{cor:strong-convexity}.
\end{proof}

The reduction requires the prescribed submanifold to contain the entire feasible set, relative interior in a proper face alone is not sufficient.

\begin{remark}[Directional scaling at boundary solutions]
The length-normalized directional-scaling condition need not hold at a boundary optimizer.  Let $\X=\operatorname{conv}\{(-1,0),(1,0),(0,1)\}\subset\R^2$ and
\[
 f(u,v)=v+\frac12u^2+\frac{\mu}{2}v^2,\qquad\mu>0.
\]
The unique constrained minimizer is $(0,0)$.  For all sufficiently small $\varepsilon>0$, $x_\varepsilon=(\varepsilon^2,\varepsilon)$ belongs to $\X$, the unique LMO endpoint is $(-1,0)$, and $g(x_\varepsilon)/\ell(x_\varepsilon)=\Theta(\varepsilon)$ while $h(x_\varepsilon)/d(x_\varepsilon,x^\star)=\Theta(1)$.  Hence no uniform $\sigma>0$ can satisfy \eqref{eq:scaling}; Proposition~\ref{prop:interior-scaling} does not extend to arbitrary boundary minimizers.
\end{remark}

\subsubsection{Loewner-constrained Karcher means}

We next specialize the interior-scaling result to the Loewner-constrained Karcher-mean problem on the affine-invariant SPD manifold.  Let $\mathbb S^d$ denote the space of real symmetric $d\times d$ matrices and define
\[
 \mathbb P_d:=\{X\in\mathbb S^d:X\succ0\},
 \qquad T_X\mathbb P_d=\mathbb S^d.
\]
For $A,B\in\mathbb S^d$, write $A\preceq B$ if $B-A$ is positive semidefinite, and $A\prec B$ if $B-A$ is positive definite.
\begin{proposition}[SPD Loewner-interval specialization]
\label{prop:spd-specialization}
Equip $\mathbb P_d$ with the affine-invariant metric
\[
 \inner{H}{K}_X:=\operatorname{tr}(X^{-1}HX^{-1}K),
 \qquad
 d(X,Y)=\norm{\log(X^{-1/2}YX^{-1/2})}_{\rm F}.
\]
Its geodesic from $X$ to $Y$ is $X\#_tY:=X^{1/2}(X^{-1/2}YX^{-1/2})^tX^{1/2}$.
For $0\prec\Xlo\prec\Xup$, let $\X=[\Xlo,\Xup]:=\{X\in\mathbb P_d:\Xlo\preceq X\preceq\Xup\}$, and consider
\begin{equation}
 f(X):=\frac12\sum_{i=1}^m w_i d(X,A_i)^2,
 \qquad A_i\in\mathbb P_d,\quad w_i>0,\quad \sum_iw_i=1.
 \label{eq:spd-centroid}
\end{equation}
Define the geodesic midpoint $X_c:=\Xlo\#_{1/2}\Xup$ and
\begin{align}
 \widehat D
 &:=\sqrt d\,\log\lambda_{\max}(\Xlo^{-1/2}\Xup\Xlo^{-1/2}),
 \label{eq:spd-diameter}\\
 R_i&:=\widehat D+d(X_c,A_i),
 \label{eq:spd-data-radius}\\
 \zeta_\kappa(r)&:=
 \begin{cases}
  \sqrt\kappa\,r\coth(\sqrt\kappa\,r),&r>0,\\
  1,&r=0,
 \end{cases}
 \qquad
 \widehat L:=\sum_{i=1}^m w_i\zeta_{1/2}(R_i).
 \label{eq:spd-smoothness}
\end{align}
Then $\mathbb P_d$ is a Hadamard manifold, $\X$ is compact and geodesically convex, and the exact RFW oracle over $\X$ admits a closed-form solution.
Moreover,
\begin{equation}
 D\leq\widehat D,
 \qquad
 f\ \text{satisfies Definition~\ref{def:strong-gconvexity} with $\mu=1$},
\label{eq:spd-certified-constants}
\end{equation}
and the quadratic upper model in Assumption~\ref{ass:regularity} is valid with smoothness parameter $\widehat L$.

Let $X^\star$ be the minimizer.
If $\Xlo\prec X^\star\prec\Xup$, then
\begin{equation}
 \beta_\star:=\min\left\{
 -\log\lambda_{\max}((X^\star)^{-1/2}\Xlo(X^\star)^{-1/2}),
 \ \log\lambda_{\min}((X^\star)^{-1/2}\Xup(X^\star)^{-1/2})
 \right\}>0
 \label{eq:spd-interior-radius}
\end{equation}
and $\overline B_{\mathbb P_d}(X^\star,\beta_\star)\subseteq\X$.  The global error bound holds with $c=\sqrt2$ and $\theta=1/2$, and the local scaling condition holds with
\begin{equation}
 \widehat\sigma=\frac{\beta_\star}{2\widehat D},\qquad
 R_S=\frac{\beta_\star}{2}.
 \label{eq:spd-rate-constants}
\end{equation}
For every $a>2$, Algorithm~\ref{alg:ol-rfw} therefore either terminates at $X^\star$ or satisfies
\begin{equation}
 f(X_k)-f^\star=O(1/k^2).
 \label{eq:spd-open-loop-rate}
\end{equation}
For the feedback short step in Proposition~\ref{prop:feedback}, 
the method likewise either terminates at $X^\star$ or satisfies
\begin{equation}
 f(X_k)-f^\star
 =O\left((1-\min\{\frac{1}{2}, \frac{\beta_\star^2}{16\widehat L\widehat D^2}\})^k\right).
 \label{eq:spd-feedback-rate}
\end{equation}
\end{proposition}
\begin{proof}
Under the stated normalization, $\mathbb P_d$  is a Hadamard manifold with sectional curvatures in $[-1/2,0]$ \citep{thanwerdas2023invariant}.
Since $\Xlo\succ0$, the interval $\X$ is a compact subset of $\mathbb P_d$.  Monotonicity of the weighted geometric mean gives
\[
 \Xlo=\Xlo\#_t\Xlo\preceq X\#_tY
 \preceq\Xup\#_t\Xup=\Xup,
 \qquad X,Y\in[\Xlo,\Xup],\quad t\in[0,1].
\]
Thus $\X$ is geodesically convex and $X_c\in\X$ \citep{bhatia2007positive}.
The closed-form oracle follows from~\citet[Theorem~4]{weber2023riemannian}.

Set $q:=\lambda_{\max}(\Xlo^{-1/2}\Xup\Xlo^{-1/2})>1$.  Then $\Xup\preceq q\Xlo$, and, for any $X,Y\in\X$,
\[
 X\preceq qY,
 \qquad Y\preceq qX.
\]
The eigenvalues of $X^{-1/2}YX^{-1/2}$ therefore lie in $[q^{-1},q]$, so $d(X,Y)\leq\sqrt d\log q=\widehat D$.

For each $i$, define
\[
 \varphi_i(X):=\frac12d(X,A_i)^2.
\]
Because $\mathbb P_d$ is Hadamard, each function $\varphi_i(X)$ is $1$-strongly geodesically convex \citep[Chapter~11]{boumal2023introduction}.
Since $\sum_i w_i=1$, the same holds for $f$.
Moreover, for every $X\in\X$,
\[
 d(X,A_i)\leq d(X,X_c)+d(X_c,A_i)
 \leq \widehat D+d(X_c,A_i)=R_i.
\]

For $X\neq A_i$, the Hessian comparison estimate in
\citet[eqs.~(2.3)--(2.5)]{afsari2013convergence}, applied with the sectional
curvature lower bound $-1/2$, gives
\[
 \operatorname{Hess}\varphi_i(X)[\xi,\xi]
 \leq
 \zeta_{1/2}\bigl(d(X,A_i)\bigr)\norm{\xi}_X^2.
\]
At $X=A_i$,
\[
 \operatorname{Hess}\varphi_i(A_i)[\xi,\xi]
 =\norm{\xi}_{A_i}^2,
\]
so the same estimate holds because $\zeta_{1/2}(0)=1$.  Since
$d(X,A_i)\leq R_i$ on $\X$ and $\zeta_{1/2}$ is nondecreasing,
\[
 \operatorname{Hess}\varphi_i(X)[\xi,\xi]
 \leq\zeta_{1/2}(R_i)\norm{\xi}_X^2.
\]
After weighting and summing, we obtain
\[
 \operatorname{Hess}f(X)[\xi,\xi]
 \leq\widehat L\norm{\xi}_X^2,
 \qquad X\in\X,\quad \xi\in T_X\mathbb P_d.
\]
For $v\in\Dset(X)$, set $\gamma(t)=\Exp_X(tv)$ and
$F(t):=f(\gamma(t))$.  Since $\gamma([0,1])\subseteq\X$ and
$\|\dot\gamma(t)\|_{\gamma(t)}=\|v\|_X$,
\[
 F''(t)
 =
 \operatorname{Hess}f(\gamma(t))
 [\dot\gamma(t),\dot\gamma(t)]
 \leq\widehat L\norm{v}_X^2.
\]
The integral form of Taylor's theorem therefore gives, for $t\in[0,1]$,
\[
 \begin{aligned}
 F(t)
 &=F(0)+tF'(0)+\int_0^t(t-s)F''(s)\,ds\\
 &\leq
 f(X)+t\inner{\grad f(X)}{v}_X
 +\frac{\widehat L}{2}t^2\norm{v}_X^2,
 \end{aligned}
\]
which is the required quadratic upper model.

Finally, let $d(X^\star,Y)\leq\beta_\star$ and set $Z=(X^\star)^{-1/2}Y(X^\star)^{-1/2}$.
Using the matrix spectral norm $\|\cdot\|_2$, we have
$
\norm{\log Z}_2\leq\norm{\log Z}_{\rm F}\leq\beta_\star,
$
and hence
\[
 e^{-\beta_\star}I\preceq Z\preceq e^{\beta_\star}I.
\]
By the definition of $\beta_\star$,
\begin{align*}
 (X^\star)^{-1/2}\Xlo(X^\star)^{-1/2}
 &\preceq e^{-\beta_\star}I\preceq Z,\\
 Z&\preceq e^{\beta_\star}I
 \preceq (X^\star)^{-1/2}\Xup(X^\star)^{-1/2}.
\end{align*}
Congruence by $(X^\star)^{1/2}$ shows that $Y\in\X$, proving the ball
inclusion.  Strong geodesic convexity gives the global HEB with
$c=\sqrt2$ and $\theta=1/2$.  Proposition~\ref{prop:interior-scaling},
together with $D\leq\widehat D$, gives
$\widehat\sigma=\beta_\star/(2\widehat D)$ and $R_S=\beta_\star/2$.
The two rates now follow from Corollary~\ref{cor:strong-convexity} and
Proposition~\ref{prop:feedback}.

\end{proof}


\section{Numerical experiments}
\label{sec:numerics}

The experiments examine the interior rate, the schedule-dependent boundary rate, and the computational cost of step-size selection.
All experiments were run in Python 3.9.12 on an Apple M1 Pro with 16 GB of memory, with BLAS restricted to one thread.

We compare three rules.  RFW-OL($a$) uses the predetermined schedule $\eta_k=a/(k+a)$.  The short step uses
$
\eta_k=\min\{1,g_k/(\widehat L\ell_k^2)\},
$
and bounded line search numerically minimizes
$f(\Exp_{x_k}(\eta v_k))$ over $\eta\in[0,1]$.  The latter is implemented
with \texttt{scipy.optimize.minimize\_scalar} using \texttt{method="bounded"}, \texttt{xatol=1e-8}, and \texttt{maxiter=60}.

For the Karcher-mean benchmark, an Armijo Riemannian gradient method computes
$X_{\rm ref}$ until
$\|\grad f(X_{\rm ref})\|_{X_{\rm ref}}\leq10^{-8}$.  Since the objective is
$1$-strongly geodesically convex,
$
0\leq f(X_{\rm ref})-f^\star \leq \frac12\|\grad f(X_{\rm ref})\|_{X_{\rm ref}}^2 =: \varepsilon_{\rm ref}.
$
Numerical primal errors are not truncated; plots and slope fits use only values above $\max\{10^{-12},10\varepsilon_{\rm ref}\}$.  Log--log slopes are obtained by least-squares fitting over the final half of these values.  In the interior experiment, the observed entry index is the first recorded $S$ after which $r(x_j)\leq\beta/2$ throughout the run.

\subsection{Strict-interior rate on the SPD manifold}
\label{subsec:controlled-experiment}

On the affine-invariant SPD manifold $\mathbb P_5$, consider the Karcher-mean problem
\begin{equation}
 \min_{\Xlo\preceq X\preceq\Xup}
 \frac12\sum_{i=1}^m w_i
 d^2(X,A_i).
 \label{eq:spd-centroid-experiment}
\end{equation}
We use sixteen equally weighted observations in eight inverse pairs
$(B_j,B_j^{-1})$, which makes $X^\star=I$ exact, and set
\[
 \Xlo=e^{-\beta}I,
 \qquad \Xup=e^{2-\beta}I,
 \qquad \beta\in\{0.1,0.5\}.
\]
Each $B_j=Q_j\operatorname{diag}(e^{\xi_j})Q_j^\top$, where $Q_j$ is Haar orthogonal and $\xi_j$ is uniform on $[-0.7,0.7]^5$.  For five fixed seeds we run OL2, OL3, OL4, and OL6 for 5000 updates when $\beta=0.1$ and 2500 updates when $\beta=0.5$.

\begin{figure}[ht]
 \centering
 \includegraphics[width=\textwidth]{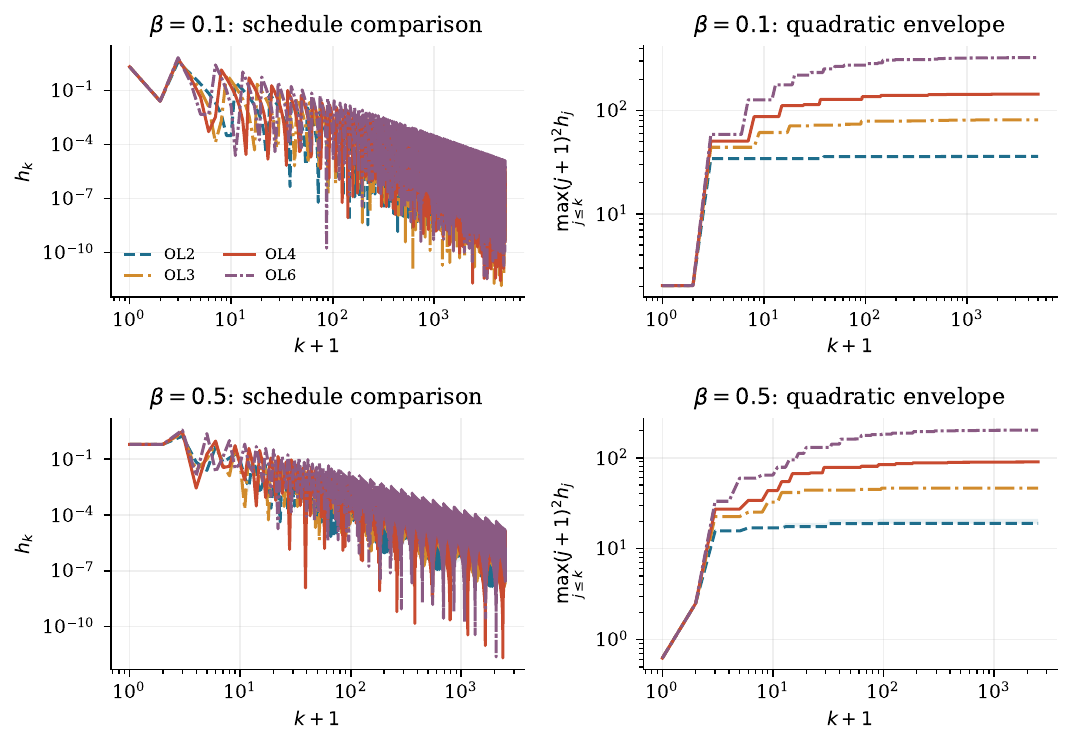}
 \caption{Strict-interior SPD instances.  Curves are medians over five seeds;
 shaded regions are interquartile ranges.
The right column shows the running envelope
$\max_{0\leq j\leq k}(j+1)^2h_j$, computed using only errors above the
reliability threshold.}
 \label{fig:theory-validation}
\end{figure}


Figure~\ref{fig:theory-validation} shows that the running quadratic envelopes
stabilize for both values of $\beta$.  The fitted slopes are close to $-2$ for
$\beta=0.1$ but less stable for $\beta=0.5$ because of stronger oscillations.
OL2 shows similar behavior, although the theorem requires $a>2$.

\subsection{Set scaling on a spherical ball}
\label{subsec:sphere-experiment}

The second experiment tests the schedule-dependent rates in Corollary~\ref{cor:sphere-general-rate} on a problem whose constants and optimizer are known explicitly:
\begin{equation}
 \min_{x\in\mathbb S^{49}} 1-x^\top y
 \quad\text{subject to}\quad d(x,o)\leq R,
 \qquad R=0.1,\quad d(o,y)=0.12.
 \label{eq:sphere-experiment}
\end{equation}
This problem corresponds to the $A=I$ case of the spherical-ball benchmark
in \citet{scieur2026strongly}, up to a constant factor. 
Here $\Delta=0.02$ and $2R+\Delta<\pi/2$.  Corollary~\ref{cor:sphere-ball} verifies $L=1$, $\nu=\sin\Delta$, and the explicit boundary optimizer.
And $\alpha_R=\tfrac12\cot(0.1)\approx4.983$.  
The oracle is reduced to the one-dimensional problem described by \citet[Section~9]{scieur2026strongly} and solved to numerical tolerance.  
We run OL2, OL3, and OL4 for 2500 updates from five fixed initial points at distance $0.08$ from $o$.

\begin{figure}[ht]
 \centering
 \includegraphics[width=\textwidth]{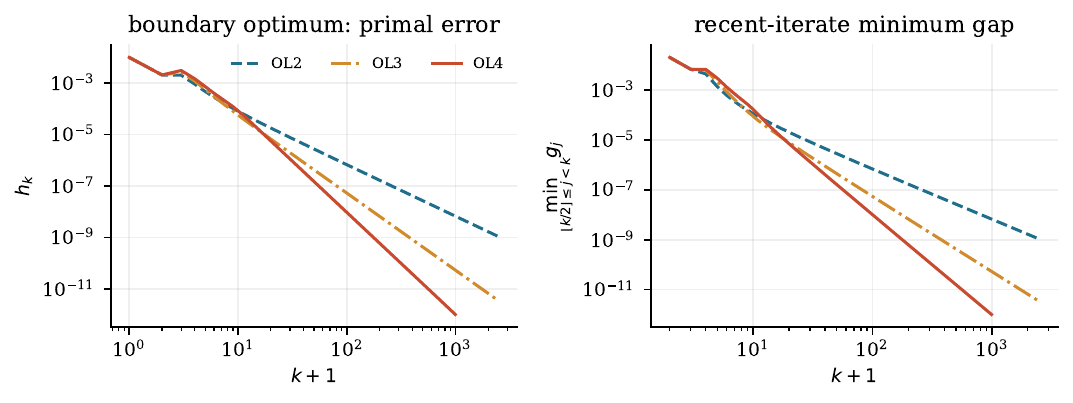}
 \caption{Sphere experiment.  The left panel shows the primal error;
 the right panel shows $g_k^{\rm rec}$, the smallest Frank--Wolfe gap among
 the most recent half of the iterates.  Curves are medians over five initial
 configurations; only values above the reliability threshold are shown.
 }
 \label{fig:sphere-scaling}
\end{figure}

The median half-tail slopes of the primal error are $-2.001$, $-3.000$, and $-3.997$ for OL2, OL3, and OL4.  The corresponding slopes of $g_k^{\rm rec}$ are $-2.002$, $-3.002$, and $-4.003$.  Thus both quantities distinguish the schedule-dependent orders in Theorem~\ref{thm:rsi-open-loop} and Corollary~\ref{cor:rsi-gap-rate}; see Figure~\ref{fig:sphere-scaling}.

\subsection{Karcher-mean cost benchmark}
\label{subsec:weber-experiment}

The last experiment returns to problem~\eqref{eq:spd-centroid-experiment} and the Loewner-interval benchmark of \citet{weber2023riemannian}.  
Riemannian means on SPD matrices arise in matrix averaging and
covariance-based classification
\citep{moakher2005differential,barachant2012multiclass}.
Each instance has ten SPD observations with logarithmic spread $s\in\{0.5,1\}$ and interval endpoints
\[
 \Xlo=\left(\frac1m\sum_{i=1}^m A_i^{-1}\right)^{-1},
 \qquad \Xup=\frac1m\sum_{i=1}^m A_i.
\]
The interval contains the unconstrained mean and has the closed-form RFW oracle used in Proposition~\ref{prop:spd-specialization}.
We use $d\in\{5,10\}$ and five seeds per $(d,s)$ group.  OL2, OL4, the short step, and bounded line search start at the interval midpoint and stop at FW gap $10^{-9}$ or after 2000 updates.

\begin{figure}[htbp]
 \centering
 \includegraphics[width=\textwidth]{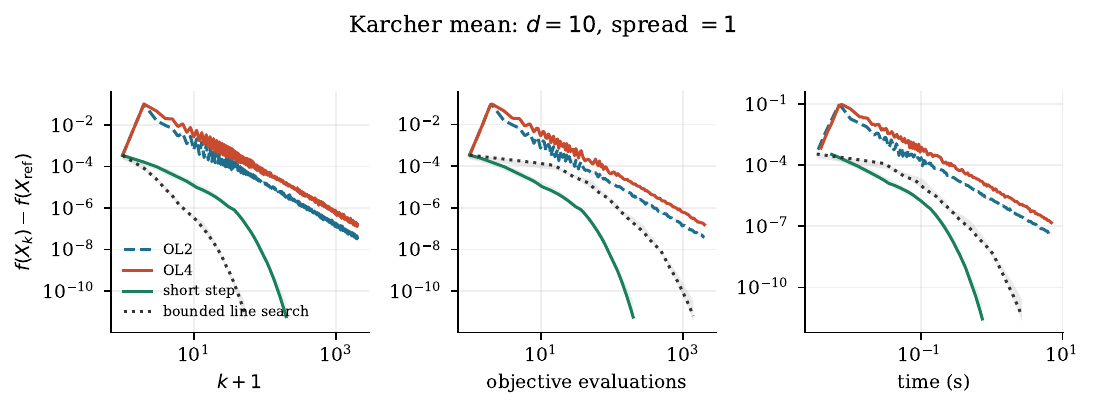}
 \caption{The Karcher-mean group ($d=10$, $s=1$), shown against
 updates, objective evaluations, and wall time.  Curves are medians over five
 seeds with interquartile bands.}
 \label{fig:weber-cost}
\end{figure}

See Figure~\ref{fig:weber-cost} for $(d=10,s=1)$.
The open-loop schemes exhibit the predicted quadratic decay but do not meet the target tolerance within the budget.  The short step gives the lowest runtime, while bounded line search reduces iterations at the expense of extra objective evaluations.  Table~\ref{tab:adaptive-results} reports the terminal statistics.

\begin{table}[H]
\centering
\caption{The $d=10$, $s=1$ Karcher-mean group.  Entries are medians over five
seeds; success means reaching FW gap $10^{-9}$ within 2000 updates.}
\label{tab:adaptive-results}
\scriptsize
\setlength{\tabcolsep}{3pt}
\begin{tabular}{@{}l r r r r r@{}}
\toprule
Method & Success & Updates & Objective evals & Final FW gap & Time (s) \\
\midrule
OL2             & 0/5 & 2000 & 2001  & $1.0\times10^{-4}$ & 7.09 \\
OL4             & 0/5 & 2000 & 2001  & $2.2\times10^{-4}$ & 7.03 \\
Short step      & 5/5 & 414  & 415   & $1.0\times10^{-9}$ & 1.45 \\
Bounded LS      & 3/5 & 1052 & 40494 & $9.2\times10^{-10}$& 72.20 \\
\bottomrule
\end{tabular}
\end{table}

\section{Conclusion}
\label{sec:conclusion}

We studied the standard exact Riemannian Frank--Wolfe method with the open loop \(\eta_k=a/(k+a)\). Under a local H\"olderian error bound and length-normalized directional scaling, the method attains the eventual rate \(O(k^{-1/(1-\theta)})\) for every \(a>2\); strong geodesic convexity and a strict-interior solution give the quadratic rate \(O(k^{-2})\). In the set-scaling regime, an exact Riemannian scaling inequality and a positive lower bound on the gradient norm yield \(O(k^{-a})\) convergence for every \(a\geq2\) after an explicit threshold index. The same order holds for the minimum Frank--Wolfe gap over the most recent half of the iterates.  Every closed spherical ball with $0<R<\pi/2$ satisfies the exact scaling inequality with $\alpha_R=\tfrac12\cot R$, so these rates apply to any objective on the ball that satisfies the standing regularity and positive gradient lower-bound conditions.

These results show that fast RFW convergence need not rely on gap-based or line-search feedback when the objective and feasible-set geometry provide sufficient control of the oracle direction. The numerical experiments recover the predicted quadratic and schedule-dependent polynomial orders and illustrate the different information requirements of predetermined and feedback-based step rules. The analysis uses exact RFW oracles and exponential-map updates, with Hadamard manifolds as the main algorithmic setting. Extending the schedule-dependent guarantees to inexact oracles or retraction-based updates remains an open direction.

\appendix
\section{An open-loop sequence lemma}
\label{app:sequence}

The following scalar result is adapted from the proof of \citet[Lemma~3.5]{wirth2023acceleration}.

\begin{lemma}[Open-loop nonlinear recurrence]
\label{lem:sequence}
Let $a>2$, let $\lambda\in(0,1)$ satisfy $\lambda a\geq2$, let $\theta\in[0,1/2]$, and let $1\leq S\leq T$.  For the auxiliary integer indices $j\geq-1$, set $\eta_j=a/(j+a)$.  Suppose $A,B,C>0$ and that sequences $(C_t)_{t=S}^{T-1}$ and $(u_t)_{t=S}^{T}$ satisfy $0\leq C_t\leq C$, $u_t\geq0$, and
\begin{equation}
 u_{t+1}\leq
 (1-\lambda\eta_t)u_t
 -\eta_t A C_t u_t^{1-\theta}
 +\eta_t^2 B C_t,
 \qquad S\leq t<T.
 \label{eq:sequence-recurrence}
\end{equation}
Then, for every $S\leq t\leq T$,
\begin{equation}
 u_t\leq\max\left\{
 \left(\frac{\eta_{t-2}}{\eta_{S-1}}\right)^{\!1/(1-\theta)}u_S,
 \left(\frac{\eta_{t-2}B}{A}\right)^{\!1/(1-\theta)}
 +\eta_{t-2}^2BC
 \right\}.
 \label{eq:sequence-bound}
\end{equation}
\end{lemma}

\begin{proof}
Set $q=1/(2(1-\theta))\in[1/2,1]$.  We first prove by induction the stronger estimate
\begin{equation}
 u_t\leq\max\left\{
 \left(
 \frac{\eta_{t-2}\eta_{t-1}}
      {\eta_{S-2}\eta_{S-1}}
 \right)^q u_S,
 \left(\eta_{t-2}\eta_{t-1}\frac{B^2}{A^2}\right)^q
 +\eta_{t-2}\eta_{t-1}BC
 \right\}.
 \label{eq:strong-sequence-bound}
\end{equation}
The case $t=S$ follows from the first branch.  Assume \eqref{eq:strong-sequence-bound} at an index $t<T$.

If
\[
 u_t\leq\left(\frac{\eta_tB}{A}\right)^{1/(1-\theta)},
\]
then $(1-\lambda\eta_t)u_t\leq u_t$, while the nonlinear term in \eqref{eq:sequence-recurrence} is nonpositive.  Using also $C_t\leq C$ and $\eta_t\leq\eta_{t-1}$ gives
\[
 u_{t+1}
 \leq u_t+\eta_t^2BC
 \leq\left(\frac{\eta_tB}{A}\right)^{1/(1-\theta)}
      +\eta_t^2BC
 \leq
 \left(\eta_{t-1}\eta_t\frac{B^2}{A^2}\right)^q
 +\eta_{t-1}\eta_tBC,
\]
which is the second branch of \eqref{eq:strong-sequence-bound} at $t+1$.

If instead $u_t\geq(\eta_tB/A)^{1/(1-\theta)}$, then $A u_t^{1-\theta}\geq\eta_tB$.  The last two terms in \eqref{eq:sequence-recurrence} have a nonpositive sum.  Moreover,
\begin{equation}
 1-\lambda\eta_t
 =\frac{t+a-\lambda a}{t+a}
 \leq\frac{t+a-2}{t+a}
 =\frac{\eta_t}{\eta_{t-2}},
 \label{eq:schedule-contraction}
\end{equation}
where the inequality is equivalent to $\lambda a\geq2$.  Consequently,
\[
 u_{t+1}\leq\frac{\eta_t}{\eta_{t-2}}u_t.
\]
Multiplying the induction bound by $\eta_t/\eta_{t-2}$ gives \eqref{eq:strong-sequence-bound} at $t+1$.  Indeed, for either branch this uses
\[
 \frac{\eta_t}{\eta_{t-2}}
 (\eta_{t-2}\eta_{t-1})^q
 \leq(\eta_{t-1}\eta_t)^q,
\]
which follows from $\eta_t/\eta_{t-2}\in(0,1]$ and $q\leq1$.  For the linear term, the corresponding relation holds with equality:
\[
 \frac{\eta_t}{\eta_{t-2}}
 \eta_{t-2}\eta_{t-1}BC
 =\eta_{t-1}\eta_tBC.
\]
This completes the induction.

Finally,
\[
 \sqrt{\eta_{t-2}\eta_{t-1}}\leq\eta_{t-2},
 \qquad
 \sqrt{\eta_{S-2}\eta_{S-1}}\geq\eta_{S-1}.
\]
Since $2q=1/(1-\theta)$, they imply
\[
 \left(
 \frac{\eta_{t-2}\eta_{t-1}}
      {\eta_{S-2}\eta_{S-1}}
 \right)^q
 \leq
 \left(\frac{\eta_{t-2}}{\eta_{S-1}}\right)^{1/(1-\theta)}
\]
and
\[
 \left(\eta_{t-2}\eta_{t-1}\frac{B^2}{A^2}\right)^q
 \leq
 \left(\frac{\eta_{t-2}B}{A}\right)^{1/(1-\theta)}.
\]
Moreover, $\eta_{t-2}\eta_{t-1}\leq\eta_{t-2}^2$.  Substitution into \eqref{eq:strong-sequence-bound} gives \eqref{eq:sequence-bound}.
\end{proof}

\begin{remark}[Why $a>2$ is required]
\label{rem:why-a-greater-two}
Within recurrence~\eqref{eq:sequence-recurrence}, the condition $\lambda a\geq2$ is necessary to guarantee the $O(t^{-2})$ rate when $\theta=1/2$.  Indeed, taking $C_t=0$ and equality in the resulting linear recurrence gives
\[
 u_t
 =u_S\prod_{j=S}^{t-1}\left(1-\frac{\lambda a}{j+a}\right)
 =u_S\frac{\Gamma(t+a-\lambda a)\Gamma(S+a)}
              {\Gamma(S+a-\lambda a)\Gamma(t+a)}
 =\Theta(t^{-\lambda a}).
\]
Thus no argument based only on this recurrence can guarantee $O(t^{-2})$ when $\lambda a<2$.  In the gap splitting used in Theorem~\ref{thm:abstract-acceleration}, choosing $\lambda=2/a$ leaves the fraction $1-2/a$ for the nonlinear scaling term, which is positive exactly when $a>2$.
When $a=2$, the condition requires $\lambda=1$, so the nonlinear scaling term disappears.  This limitation is specific to the recurrence analysis and does not rule out quadratic convergence of OL2 on particular problems.
\end{remark}

\bibliography{references}

\section*{Statements and Declarations}

\noindent\textbf{Funding} No funding was received to assist with the preparation of this manuscript.

\noindent\textbf{Competing Interests} The author has no relevant financial or non-financial interests to disclose.

\noindent\textbf{Data Availability Statement} No datasets were generated or analyzed in this study.

\end{document}